\documentclass[11pt]{article}
\usepackage{graphicx} 
\usepackage{amsmath}
\usepackage{amssymb}
\usepackage{amsthm}
\usepackage{mathtools}
\usepackage{lipsum}
\usepackage{amsfonts}
\usepackage{mwe}
\usepackage{tikz}
\usepackage{url}
\usepackage{subcaption}
\usepackage{authblk}
\usepackage[normalem]{ulem}

\usepackage[numbers,sort&compress]{natbib}

\usetikzlibrary{matrix,positioning,fit,decorations.pathreplacing}

\usepackage[left=3cm, right=3cm]{geometry}
\newtheorem{definition}{Definition}

\newtheorem{lemma}{Lemma}
\newtheorem{theorem}{Theorem}
\newtheorem{remark}{Remark}

\newcommand{\Hi}{\ensuremath{H_{\mathrm{I}}}}
\newcommand{\Hp}{\ensuremath{H_{\mathrm{P}}}}
\newcommand{\Ai}{\ensuremath{A_{\mathrm{I}}}}
\newcommand{\An}{\ensuremath{A_{\mathrm{N}}}}
\newcommand{\Pa}{\ensuremath{P_{\mathrm{A}}}}
\renewcommand{\Pi}{\ensuremath{P_{\mathrm{I}}}}
\newcommand{\Dt}{\ensuremath{D_{t}}}
\newcommand{\Dtn}{\ensuremath{D_{t,\mathrm{N}}}}
\newcommand{\betai}{\ensuremath{\beta_{\mathrm{I}}}}
\newcommand{\betan}{\ensuremath{\beta_{\mathrm{N}}}}
\newcommand{\sym}{\ensuremath{\mathrm{sym}}}
\newcommand{\esym}{\varepsilon_{\sym}}
\DeclareMathOperator{\Tr}{\mathrm{Tr}}
\DeclarePairedDelimiterX{\norm}[1]{\lVert}{\rVert}{#1}
\DeclarePairedDelimiterX{\dotp}[2]{\langle}{\rangle}{#1,#2}
\def\eqdef{\overset{\text{def}}{=}}
\DeclareMathOperator{\Span}{Span}

\title{The High Cost of Data Augmentation \\ for Learning Equivariant Models}
\author[1]{Henri Klintebäck}
\author[1]{Christoph Ortner}
\author[1]{Lior Silberman}
\affil[1]{1984 Mathematics Rd, Vancouver, Department of Mathematics, University of British Columbia\\
\texttt{klintebaeck@math.ubc.ca}}

\date{\today}

\begin{document}

\maketitle

\begin{abstract}
According to Noether's theorem the presence of a continuous symmetry in a Hamiltonian systems is equivalent to the existence of a conserved quantity, yet these symmetries are not always explicitly enforced in data-driven models. There remains a debate whether or not encoding of symmetry into a model architecture is the optimal approach. A competing approach is to target approximate symmetry through data augmentation. In this work, we study two approaches aimed at improving the symmetry properties of such an approximation scheme: one based on a quadrature rule for the Haar measure on the compact Lie group encoding the continuous symmetry of interest and one based on a random sampling of that Haar measure.
We demonstrate both theoretically and empirically that the quadrature augmentation leads to exact symmetry preservation in polynomial models, while the random augmentation has only square-root convergence of the symmetrization error.
\end{abstract}

\section{Introduction}

The application of machine learning (ML) to the development of surrogate models of complex physical systems is an active area of research.
This is due to the intrinsic nature of physical systems which are often so complex that conventional surrogate models fail to describe them adequately. Prolific research areas for {\it machine learning surrogates} include quantum-accurate force fields for molecular dynamics \citep{PhysRevLett.98.146401}, 
coarse-grained molecular dynamics \citep{ML-CGMD}, and more generally of learning dynamical systems from data~\citep{YuOverview}. 

The focus of the present work is on the role of symmetries in machine learning model architectures. These can be found everywhere in the physical, chemical and biological world, be it in atomic potentials, molecular structures or protein interactions.
Two competing approaches to handling symmetry in machine learning models are (i) using a generic model that is unaware of symmetry but augmenting training data during learning \citep{pmlr-v162-brandstetter22a,MonteCarloRef}; (ii) exactly enforcing known symmetry in the model \citep{MACE, JMLR:Reynolds}. Each possibility can be achieved through a number of available techniques. It remains an ongoing debate (cf. foregoing references) which of these two approaches are preferable, and how they are to be optimally implemented. Generally speaking, striking a balance between enforcing structure in model-design and learning it from data remains a challenge in scientific machine learning. The purpose of our work is to rigorously analyze two natural data augmentation approaches to learn a known model symmetry and to contrast them against each other as well as against an analogous model with built-in symmetry.

An argument often advanced in favour of hard-coding symmetry into the model architecture is that it reduces the number of model parameters. This is generally true, but in the context of deep learning, the number of parameters is rarely an adequate measure of a model's performance or robustness. The properties of training dynamics is generally more important and some evidence is accumulating that symmetry-unconstrained model may indeed be easier to train and are equally robust and accurate out-of-distribution \citep{Smidt25,MCN25}. This has already lead to the development of highly efficient unconstrained models for static tasks such as geometric optimization \citep{Ceriotti26, CeriottiUniversal}

Our concern with preserving symmetry instead focuses on conserving important qualitative physics of a model. According to Noether's theorem, symmetries and conserved quantities are intrinsically linked; the presence of a continuous symmetry is equivalent to the existence of a first integral of the motion. For example, translation invariance results in conservation of linear momentum while rotation invariance results in conservation of angular momentum. For many applications, it is desired that these conserved quantities are preserved by the numerical scheme.
This led, for example, to the development of geometric integration, in particular symplectic integration \citep{HLW}. A symplectic integrator is an integrator such that its time-stepping operation is a symplectomorphism, a consequence being that any first integrals of the simulated system will be preserved \citep{Horikoshi} for a separable Hamiltonian.

In section \ref{section:AngMom}, we will show that this the impact of having only {\it approximate} symmetry in a surrogate model generically has a substantial impact on the conservation of the first integrals. We are particularly interested in the case of invariance under $\mathrm{SO}(n)$ whose associated conserved quantity in two and three dimensions is angular momentum. We review established theory explaining that --- except in the case of fully integrable systems --- one will observe an exponential drift in the time evolution of these conserved quantities. This motivates the need for surrogate models models to preserve conserved quantities exactly or to within very high precision. 

In section \ref{sec:DataAug} we then present and study two different methods of data augmentation in order to improve the symmetry properties of a symmetry-unconstrained approximation scheme. These are augmenting the data through a quadrature rule for the Haar measure on the compact Lie group whose representation is the continuous symmetry of interest and a random sampling of that Haar measure. We rigorously prove convergence (rates) for both of these schemes. Finally, in section \ref{sec:NumTest}, we provide a more empirical analysis to illustrate our theory and provide further details of the performance of data augmentation schemes approximation schemes through a Least-Squares toy model - we will be considering the simplest non-trivial case of three particles on $S^1$ and $S^2$.
In section \ref{sec:DataAug} we then present and study two different methods of data augmentation in order to improve the symmetry properties of a symmetry-unconstrained approximation scheme. These are augmenting the data through a quadrature rule for the Haar measure on the compact Lie group whose representation is the continuous symmetry of interest and a random sampling of that Haar measure. We rigorously prove convergence (rates) for both of these schemes. 

Finally, in section \ref{sec:NumTest}, we provide  empirical numerical tests to illustrate our rigorous theory and to provide further details of the performance of the data augmentation schemes.

\section{Approximate conservation of angular momentum}
\label{section:AngMom}
In this section, we review results concerning approximate conservation properties of symplectic integrators and their breakdown when a system is only approximately symmetric. $H=T+V$ denotes a separable Hamiltonian and $(p,q)$ the canonical coordinates in phase space. This section gives a formal discussion of key ideas; we refer to \citep{HLW} for technical details. 

A numerical integrator for a Hamiltonian system is symplectic if it is compatible with the manifold structure; cf. \citep{Arnold1989, HLW}.
The canonical example of a symplectic integrator is the order 2 Velocity Verlet integrator we chose for our numerical applications \citep{VelocityVerlet}. 
While the dynamics obtained through the integration scheme are not exact, they correspond to a Hamiltonian that is close to the original one, known as the shadow Hamiltonian $H_{\mathrm{eff}}$. For any symplectic integrator of order $p$ with stepsize $h$
%
\begin{equation}
    H_{\mathrm{eff}}=H+O(h^p). 
\end{equation}

One of $H_{\mathrm{eff}}$'s constructions may be found in \citep{Horikoshi}. It implies that any quantity that commutes with both $T$ and $V$ also commutes with $H_{\mathrm{eff}}$ as all its higher order terms are commutators of $T$ and $V$.  
In practice this is a strong result on the conservation of first integrals, as $J$ is a first integral iff $\{J,H\}=0$. Thus, $H$ and $H_{\mathrm{eff}}$ share the same conserved quantities.

We are interested in the case of a rotation-invariant potential $V$, resulting in conservation of angular momentum under Hamiltonian dynamics. More generally, we can consider continuous symmetries.

Suppose now, that during an approximation procedure (see later section) exact invariance is violated, but the resulting Hamiltonian remains approximately invariant, that is, $H=\Hi+\varepsilon \Hp$, where $\Hi$ has a continuous symmetry, $\varepsilon > 0$ and $\Hp$ is a Hamiltonian.  In this case, approximate conservation properties are not guaranteed. Two scenarios best illustrate the expected behavior: either trajectories stay controlled or the system contains some amount of chaos. The general setting belongs to the latter case; even  under analytic assumptions, we observe an exponential growth in the error. This is expected, as it is well-known that if a system is not fully integrable, then it is likely for chaos to be present \citep{Lyapunov1, Lyapunov2, LyapunovAS1,LyapunovAS2, LyapunovAS3}.

Although symplectic integrators do not alter this behaviour, they do guarantee that conserved quantities are preserved correctly. 
To observe this, we substitute $X_{\mathrm{V}} = X_{\mathrm{i}}+\epsilon X_\mathrm{p}$ into $X_{\mathrm{eff}}$'s expression, where $X_{\mathrm{i}}$ respectively $X_{\mathrm{p}}$ are the Liouville operators \citep{Horikoshi, Arnold1989} for $H_{\mathrm{i}}$ respectively $H_{\mathrm{p}}$. This yields

\begin{align}
    X_{\mathrm{eff}} &= X_{\mathrm{T}}+ X_{\mathrm{i}}+\epsilon X_{\mathrm{p}} + \frac{h}{2}[X_{\mathrm{T}},X_{\mathrm{i}}]+\frac{h\epsilon}{2}[X_{\mathrm{T}},X_{\mathrm{p}}]+\cdots\\
    J\circ \Phi_t&= e^{tX_{\mathrm{eff}}}J\\
    &= e^{t\epsilon X_{\mathrm{eff,p}}}e^{t(X_{\mathrm{T}}+X_{\mathrm{i}})}J\\
    &=e^{t\epsilon X_{\mathrm{eff,p}}}J,
\end{align}
where $\Phi_t$ is the Hamiltonian flow, that is the map such that $\Phi_t(p(0),q(0))=(p(t),q(t))$.
We remark that $e^{t\epsilon X_{\mathrm{eff,p}}}J=J$ iff $X_{\mathrm{p}}$ commutes with $J$.

While exponential deviation is the generic situation, more rigid behaviour appears in integrable systems, those with $n$ (the dimension of configuration space) functionally independent constants of the motion. KAM theory \citep{KAMRef} shows that in a completely integrable system, sufficiently irrational non-degenerate tori survive perturbation, implying that conserved quantities will be approximately conserved in the perturbed system.
This is shown in Chapter XI of \citep{HLW}.

We illustrate this behaviour through the following simulations. We consider 3 particles on the circle $S^1$ under a potential $U=F+\epsilon G$, where $F$ is rotation-invariant and analytic, while $G$ is chosen to be analytic but not invariant under rotations. Due to Noether's theorem, a continuous symmetry in the potential is equivalent to the presence of a first integral; rotation invariance corresponds to the conservation of angular momentum. The initial conditions are such that the initial total angular momentum vanishes. We perform the simulations for different values of $\epsilon$. 
To compare approximately symmetric potentials,  we define the symmetrization error.

\begin{definition}
    Let $G$ be a compact Lie Group, $H$ be the Haar measure on $G$. 
    Let $f$ be a function $f\colon V\to\mathbb{C}$, where $G$ is acting onto a vector space $V$. Let $\sym$ be the orthogonal projection, with respect to the $L^2$ inner product, onto the space of $G$ invariant functions. Then, we define the symmetrisation error to be
    \begin{equation}
        \epsilon_{\sym}(f)\eqdef||f-\sym(f)||_{L^2}.
    \end{equation}
\end{definition}

When $f$ is $G$-invariant $\epsilon_{\sym}(f)=0$. The choice of norm can in general have strong consequences. However, in our applications we consider only very smooth function classes such that the $L^2$ and $L^{\infty}$ norms remain comparable up to moderate constants.

\begin{figure*}[!htb]
    \centering
    \begin{subfigure}[t]{0.45\textwidth}
        \centering
        \includegraphics[width=\textwidth]{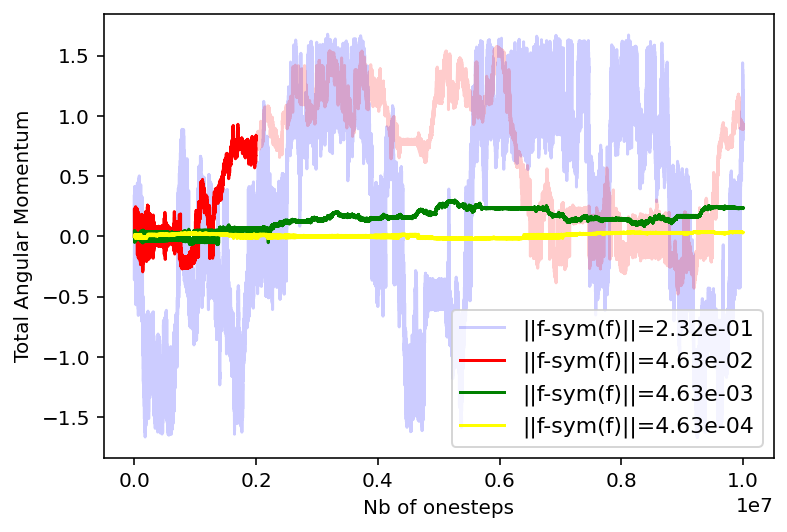} 
        \caption{The time evolution of the total angular momentum of a system of 3 particles on $S^1$ under the previously described potentials. Simulations are run for 10 million time steps.}
        \label{fig:ApproxJ2dA}
    \end{subfigure}%
    ~
    \begin{subfigure}[t]{0.45\textwidth}
        \centering
        \includegraphics[width=\textwidth]{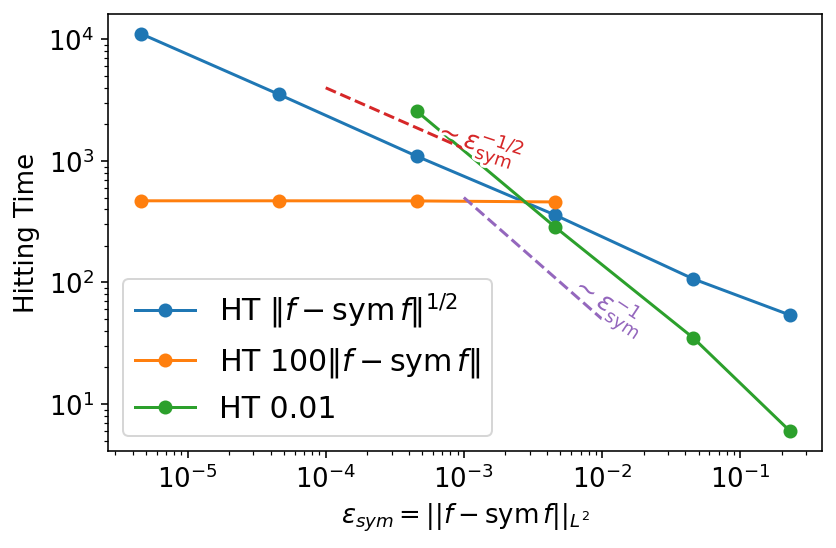} 
        \caption{The hitting times of different error targets $\varepsilon$. The hitting time corresponds to the first time step $t_n$, such that $|J(t_n)|\geq\varepsilon$}.
        \label{fig:ApproxJ2dB}
    \end{subfigure}
    \caption{All computations are performed with a Verlet integrator of step-size $dt = 0.05$.}
    \label{fig:ApproxJ2d}
\end{figure*}

Figure $\ref{fig:ApproxJ2d}$ illustrates the rapid breakdown of the conservation of angular momentum for approximately rotation-invariant potentials. Thus, if one wishes to preserve conserved quantities, either one should be working with a fully integrable system or have an approximation scheme that exactly preserves  symmetries. 

To conclude, we motivated that except in the very fringe case of full integrability, the Hamiltonian one gives to their numerical scheme should preserves symmetries as exactly as possible. Indeed, outside of KAM theory, only an exactly symmetric Hamiltonian will provably always have simulated dynamics whose outputs preserve the desired first integrals.

\subsection{Geometry Optimization}

Our foregoing discussion makes an argument that symmetry must be preserved to within very high precision in order to prevent loss of important conservation laws and the resulting (potentially catastrophic) failure of molecular dynamics simulations. Whether such failure occurs would depend on the specific simulation undertaken, e.g., the importance of conservation of angular momentum is different in a gas, liquid or solid. 

In addition, this does not prevent a non-symmetry preserving model from being well suited for other simulation tasks. For example, in ``geometry optimization'' (computation of energy minima, saddle points, etc)  we generally seek an equilibrium point of $V$, i.e., solve the nonlinear system 
\[
    \nabla V(\bar{q}) = 0.
\]
Suppose now that an approximate model $V_\epsilon$ exists such that $|V - V_\epsilon| \leq \epsilon^2$, and $\|\nabla V - \nabla V_\epsilon \|_2 = \epsilon$ in a neighbourhood of such an equilibrium point $\bar{q}$. Under suitable and natural additional assumptions on the stability of the equilibrium point, the inverse function theorem guarantees existence of an approximate equilibrium $\bar{q}_\epsilon$ such that 
\[
    \nabla V_\epsilon(\bar{q}_\epsilon) = 0, 
    \qquad 
    \| \bar{q} - \bar{q}_\epsilon \|_2 \leq C \epsilon 
    \quad \text{and} \quad 
    |V(\bar{q}) - V_\epsilon(\bar{q}_\epsilon)| \leq C \epsilon^2.    
\]
Obtaining this conclusion required no information about the symmetry of $V_\epsilon$. Thus, the requirement to exactly (or, to within high precision) preserve symmetries is necessarily highly task dependent. This is consistent with numerous empirical observations on non-trivial benchmark tasks~\citep{CeriottiUniversal}.

\section{Analysis of Data Augmentation}
\label{sec:DataAug}
We first summarize the notation used throughout the rest of the paper. $\mathbb{N}$ are the natural numbers including 0. $S^{d},d\in\mathbb{N}\setminus\{0\}$ is the $d$-sphere and $S^{N,d}$ its $N$-fold product. We denote by $L^2(S^{N,d})$ the space of complex-valued, measurable, square integrable functions. $R = (\mathbf{r}_j)_{j=1}^N\in S^{N,d}$ is a collection of points on the general $d$-sphere, for the special case of $d=1$, we denote by $\Theta=(\theta_1,\cdots,\theta_N)\in S^{N,1}$ a collection of angles representing points on the unit circle.  
The symbols $\mathbf{l},\mathbf{m},\mathbf{k}$ denote multi-indices, that is, tuples of integers, for instance $\mathbf{k}=(k_1,\cdots,k_N)\in\mathbb{Z}^N$. $Q\in \mathrm{SO}(d+1)$ is a rotation. The action of $Q$ on $\Theta$ or $R$ is written as $Q\cdot \Theta$ or $Q \cdot R$ and stands for element-wise rotation, i.e., 
\begin{equation}
    Q\cdot R = (Q\mathbf{r}_1,\cdots, Q\mathbf{r}_N).
    \label{eq:defAction}
\end{equation}

\subsection{Polynomial Regression}
We study different ways of approximating (or, ``learning'') a rotation-invariant potential $f\colon S^{N,d}\to\mathbb{C}$. 
We will focus on approximations using polynomial tensor product bases. Their elements are, 
\begin{align}
    \varphi_{\mathbf{k}}(\theta_1,\cdots, \theta_N)&=\prod_{j=1}^N e^{ik_j\theta_j}, \qquad \text{ for $d=1$}, \\
    \varphi_{\mathbf{l,m}}(\mathbf{r}_1,\cdots, \mathbf{r}_N)&=\prod_{i=1}^N Y^{m_i}_{l_i}(\mathbf{r}_i), \qquad \text{ for $d=2$},
    \label{eq:tensorbasis}
\end{align}
where the $Y^{m}_{l}$ are complex spherical harmonics: for $l\in\mathbb{N}$, $-l\leq m\leq l$, $(\theta,\varphi)\in S^2$, they are defined by
\begin{align*}
    Y^m_l(\theta,\varphi) &= (-1)^m\sqrt{\frac{(2l+1)(l-m)!}{4\pi(l+m)!}}P^m_l(\cos(\theta))e^{im\varphi}, \\ 
    P^m_l(x) &= \frac{(-1)^m}{2^l l!}(1-x^2)^{\frac{m}{2}}\frac{d^{l+m}}{dx^{l+m}}(x^2-1)^l.
\end{align*}
In practice, this definition is never used as we rely on numerically stable recursive evaluation algorithms~\cite{Scipy}. These functions are polynomials in the cartesian coordinates of $\mathbb{R}^{d+1}$. 
We denote by
\begin{equation}
    V^{N,d}_{K}\subset L^2(S^{N,d})
    \label{eq:defV}
\end{equation}
the subspace of polynomial functions (in cartesian coordinates) on $S^{N,d}$ of {\em total degree} at most $K$; specifically,  
\begin{align}
    V_K^{N,1} &:= \Span\left\{ \varphi_{\mathbf{k}}\;\colon \; \mathbf{k}\in \mathbb{Z}^N, \norm{\mathbf{k}}_1\leq K \right\}, \, \quad \text{and} \\
    V^{N,2}_{K} &:= \text{Span}\left\{\varphi_{\mathbf{l},\mathbf{m}}\colon {\bf l} \in \mathbb{N}^N, {\bf m} \in \mathbb{Z}^N, |m_i| \leq l_i, 
    \| {\bf l} \|_1 \leq K \right\}.
\end{align}
The total degree of $\varphi_{\bf l \bf m}$ is taken to be $\| {\bf l}\|_1$ to ensure that $V^{N,2}_{K}$ is closed under rotations (\citep{BookRepTh}; see also \S~\ref{subsec:32}).
It is well-known that $\bigcup_{K} V^{N,d}_K$ - the space of all polynomials - is dense in $L^2$~\citep{bump2004lie}.

For notational convenience we specify an arbitrary ordering, ${\bf k}_j$ or $({\bf l}_j, {\bf m}_j)$ of the basis elements and identifying $\varphi_j = \varphi_{{\bf k}_j}$ resp. $\varphi_j = \varphi_{{\bf l}_j {\bf m}_j}$. 
An arbitrary polynomial $P \in V_K^{N, d}$ can then be written in the form
\begin{equation}
    P_{\beta} = \sum_{j}\beta_{j}\varphi_{j},
\end{equation}
with parameters $\beta_{j}$. 

We employ a Least-Squares Approximation scheme (LSQ) to construct approximants. Let $(R_i, f(R_i))_{i=1}^n$ be iid data sampled with respect to some distribution (they are specified in Section \ref{sec:NumTest}). From it, we  construct the design matrix $A$, $A_{ij}=(\varphi_{j}(R_i))$, of degree $K$ and the data vector $Y$, with entries $Y_i = f(R_i)$.
The LSQ parameters are then determined by solving the linear least squares problem
\begin{equation}
\beta = \mathrm{argmin}_{\beta}L(\beta), 
\qquad \text{where} 
\qquad
L(\beta)=\frac{1}{2}\norm{A\beta - Y}^2.
\label{eq:fullLSQ}
\end{equation}
It is well known that under the assumption that there is enough data, the LSQ approximates the $L^2$ projection onto $V^{N,d}_K$ \citep{györfi2002}. However, no additional information is given about the symmetry of the estimated polynomial, and it is {\it a priori} unclear whether the deviation from exact invariance is such that it could lead to dramatic qualitative failures of a simulation (cf. Figure \ref{fig:ApproxJ2d}).

\subsection{Regression in Invariant Subspaces}
\label{subsec:32}
Arguably the canonical approach to learning an invariant function is restricting the approximation to the subspace of rotation-invariant polynomials on $S^{N,d}$ of total degree at most $K$: 
\begin{equation}
    B^{N,d}_K := \big\{P\in V^{N,d}_K\colon P\circ Q=P,\forall Q\in \mathrm{SO}(d+1) \big\} \subset V^{N,d}_K,
\end{equation}
where $V^{N,d}_K$, defined in \eqref{eq:defV}, is the space of polynomials of total degree at most $K$ on $S^{N,d}$ and the rotations have elementwise \eqref{eq:defAction} action.

\begin{lemma}
\label{lemma:1}
    Let $H$ be the normalised Haar measure on $G=\text{SO}(d+1)$ and let $f\in L^2(S^{N,d})$. Then,
    \begin{equation}
        \sym(f) = \int_G f\circ QH(dQ).
        \label{eq:symint}
    \end{equation}
    Furthermore, the symmetrisation error $\esym(f)$ satisfies
    \begin{equation*}
        \esym(f) \eqdef \norm{f-\sym(f)}_{L^2}\leq \norm{f-\bar{f}}_{L^2}
    \end{equation*}
    for any $G$-invariant function $\bar{f}$.
\end{lemma}

The proof of the Lemma is standard; we include it for the reader's convenience.

\begin{proof}
    For an $L^2$ function $f$, we define $\pi(f)$ to be its group  average in \eqref{eq:symint}. We will show that it is the $L^2$ orthogonal projection into $L^2_G$ - the space of rotation-invariant $L^2$ functions. First remark that $\pi(f)$ is rotation invariant, as composing with a rotation will simply shift the integral: let $Q'\in G$, then
    \begin{equation}
        \pi(f)\circ Q'=\int_G f\circ (QQ')H(dQ) = \pi(f).
    \end{equation}
    Thus, $\mathrm{Im}(\pi)\subset L^2_G$.
    Changing variables it is straightforward to see that $\pi(\pi(f))=\pi(f)$ and that $\pi(\bar{f})=\bar{f}$ if $\bar{f}$ is rotation-invariant, in other words that $\pi$ is a projection onto the space of invariant functions.  We verify that it is self-adjoint, which would allow us to conclude that it is an orthogonal projector. For any function $g\in L^2$,
    \begin{align}
        \dotp*{\pi(f)}{g} & = \dotp*{\int_G  (f\circ Q)H(dQ)}{g} \\
        & = \int_G \dotp*{f\circ Q}{g}H(dQ) \\
        & = \int_G \dotp*{f}{g\circ Q^{-1}} H(dQ)\\
        & = \int_G \dotp*{f}{g\circ Q}H(dQ) \\
        & = \dotp*{f}{\pi(g)}\,.
    \end{align}
    Thus, $\pi\colon L^2\to L^2_G$ is the $L^2$-orthogonal projector and we conclude that $\pi=\sym$.
    The inequality then follows since $\sym(f)=\mathrm{argmin}_{\bar{f}\in L^2_G}(\norm*{f-\bar{f}})$.
    This is well defined as we claim that $L^2_G$ is closed. Indeed, let $\rho\colon G\to U(L^2)$ be our unitary representation - that is the above composition of a function with a rotation (cf. \eqref{eq:defAction}). Then, 
    \begin{equation}
        L^2_G = \{f\in L^2\colon \rho(Q) f= f, \forall Q\in G\} = \bigcap_{Q\in G}\ker(\rho(Q)-\mathrm{Id}).
    \end{equation}
    The conclusion follows as an arbitrary intersection of closed sets is closed.
\end{proof}

Using Lemma \ref{lemma:1} we can characterise the rotation-invariant subspaces. For $d=1$,
\begin{align}
        B_K^{N,1} & = \mathrm{Span}\left\{\varphi_{\mathbf{k}} \;\colon\; \varphi_{\mathbf{k}}\in V^{N,1}_K,\sum_{i=1}^N k_i=0\right\}.
\end{align}
This can be shown by applying $\sym$ to the basis elements of $V^{N,1}_K$.
\begin{align}
    \notag 
    \sym(\varphi_{\mathbf{k}})
    &= 
    \frac{1}{2\pi}\int_{0}^{2\pi}\prod_{j=1}^Ne^{ik_j(\theta_j+\alpha)}d\alpha \\ 
    &=\frac{1}{2\pi}\int_{0}^{2\pi}e^{i\alpha\sum_j k_j}\prod_{j=1}^Ne^{ik_j\theta_j}d\alpha 
    = 
    \begin{cases}
        \varphi_{\bf k}, & \sum_i k_i = 0, \\ 
        0, & \text{otherwise.}
    \end{cases}.
\end{align}

There are multiple natural choices of a basis for the rotation-invariant subspace of $V^{N,2}_K$. Possible approaches include applying \eqref{eq:symint} directly, or to recursively decompose tensor products of pairs of irreducible representations using Clebsch--Gordan coefficients. Any approach results in invariant basis functions
\begin{equation}
    \varphi_{\mathbf{l},\mu}=\sum_{\mathbf{m\in\mathcal{M}_{\mathbf{l}}}}\mathcal{C}^{\mathbf{l}\mu}_{\mathbf{m}}\varphi_{\mathbf{l}\mathbf{m}},
\end{equation}
where $\mu$ encodes all the possible invariant couplings and $\mathcal{C}^{\mathbf{l}\mu}_{\mathbf{m}}$ are generalized (higher-order) Clebsch-Gordan coefficients.
The details can be found in \citep{BookRepTh} and \citep{ACE}.  

All our numerical tests are performed with $N=3$ particles, in which case a basis of $B^{3,2}_L$ is given by 
\begin{align}
    \tilde{\varphi}_{\mathbf{l}}(R) = \sum_{\substack{\boldsymbol{\mu}\in\mathcal{M}_{\mathbf{l}}\\ \mu_1+\mu_2+\mu_3=0}} C^{l_3,0}_{l_1,0,l_2,0}C^{l_3,\mu_1+\mu_2}_{l_1,\mu_1,l_2,\mu_2}C^{00}_{l_3,0,l_3,0} C^{00}_{l_3,-\mu_3,l_3,\mu_3}\prod_{i=1}^3 Y^{m_i}_{l_i}(\mathbf{r}_i), & \\ 
    \text{where } \quad |l_1-l_2|\leq l_3 \leq l_1+l_2.&
    \label{eq:3dIbasis}
\end{align}

We can now define the invariant design matrix, $(\Ai)_{ij}=\tilde{\varphi}_{j}(x_i)$ where we have similarly specified an ordering of the basis elements of $B^{N,d}_K$. The invariant regression problem is 
\begin{equation}
    \bar\betai = \mathrm{argmin}_{\betai} L_{\mathrm{I}}(\betai) 
    \quad \text{where} \quad 
    L_{\mathrm{I}}(\betai)=\frac{1}{2}\norm{\Ai\betai -Y}^2 \; .
    \label{eq:invlsq}
\end{equation}

This idea of enforcing structure in the architectures is widespread in the ML community. For instance, Physics Informed Neural Networks \citep{PINNS} are neural networks that force their outputs to be solutions to the equations they are trying to learn. In the case of symmetries, enforcing such a constraint in a ML model leads to the so called Equivariant Neural Networks (ENNs) \citep{ENN}. These are widely used to study atomic potentials \citep{MACE, ACE} or other domains where symmetries are widespread such as lattice gauge theory \citep{Favoni_2022}. Applications are not solely restricted to physics, for example, ENNs are also used in image processing \citep{bengtssonbernander2024} as two images that differ by some rigid body transformation are usually thought to be the same.

\subsection{Data Augmentation}
Enforcing symmetry will always result in exact symmetry of the learned model. Interestingly, however, it may not always be the optimal choice. In fact, recent studies have shown that enforcing symmetry may hurt performance \citep{bökman23} or make training harder \citep{MCN25}. To avoid these possible issues, one may want to chose another way of learning symmetry, for instance through data augmentation.
For $T\in\mathbb{N}$, let $(w_t, Q_t)_{t=1}^T$ be a collection of real weights and rotations. We will specify below how they are obtained. Then the symmetry-augmented linear least squares problem is
\begin{equation}
    \beta_{\rm aug} = \arg\min L_{\rm aug}(\beta) 
    \qquad \text{where} \qquad 
     L_{\mathrm{aug}}(\beta) =\frac{1}{2}  \sum_{t=1}^Tw_t \big\|(A\circ Q_t)\beta-Y \big\|^2,
     \label{eq:lsqaug}
\end{equation}
where $A\circ Q_t$ denotes the design matrix evaluated at the points $Q_t\cdot R$, $t = 1, \dots, n$. The notation $A\circ Q$ comes from the fact that we may think of the design matrix as a function from data space to the space of matrices. The design matrix obtained from rotating data by $Q_t$ then corresponds to a composition with $Q_t$. Remark that $Y$ does not need to be rotated as the target-function is chosen to be rotation-invariant. 

This idea stems from the above observation that for a rotation-invariant function $f$, its value at one point $R$ is constant on the whole orbit $G\cdot R$. 
We present two ways of augmenting data: a random and an optimal sampling of $G$, leading to a near square root decay in $\esym$ respectively exact invariance. An important observation is that this type of augmentation does not correspond to a simple symmetrisation of the solution. The dependence of the output of our augmented LSQ system (cf \eqref{eq:lsqaug}) is nonlinear and the effects of augmentation must carefully analysed.

It is well known that $V^{N,d}_K$ is closed under rotations. That is, for any $v\in V^{N,d}_K$ and $Q\in \mathrm{SO}(d+1)$, $v\circ Q\in V^{N,d}_K$.
Moreover, this action is linear on the space of polynomials. They are related with each other through the (unitary) Wigner $D$-matrices \citep{BookRepTh} denoted by $D$. 
If we consider the spherical harmonic $Y^m_l$, then its rotation by $Q$ may be written as
\begin{equation}
    Y^m_l\circ Q = \sum_{m'=-l}^lY^{m'}_l D^l_{m'm}(Q).
\end{equation}
\begin{lemma}
    For any $Q\in\mathrm{SO}(d+1)$, there exists a unitary matrix $\mathcal{D}(Q)$ such that
\begin{equation}
    A\circ Q = A\mathcal{D}(Q).
\end{equation}
    \label{lemma:2}
\end{lemma}
\begin{proof}
    We can generalise the notion of Wigner D-matrices to our polynomial tensor basis, as the tensor product of representations is a representation. Moreover, these matrices stay unitary as the finite dimensional representations of compact Lie groups are unitary \citep{bump2004lie}. Thus, we know the existence of $\mathcal{D}(Q)$ given $Q\in \mathrm{SO}(d+1)$ such that for all $v\in V^{N,d}_K$, $v\circ Q = v\mathcal{D}(Q)$. Since a design matrix is a concatenation of such vectors, the above $\mathcal{D}$ satisfies the desired properties.
\end{proof}
These matrices $\mathcal{D}$ can be thought of as generalised $D$-matrices. We refer to \citep{BookRepTh,ACE} for details, but we include an explicit computation for the case of particles on the circle ($d=1$) for the sake of illustration and intuition. We recall that a basis element $\varphi_{\mathbf{k}}$ is given by a product of trigonometric polynomials \eqref{eq:tensorbasis}. Acting on it by a rotation $Q_{\alpha}$ of angle $\alpha$, we obtain
\begin{align}
    Q_{\alpha}\cdot \varphi_{\mathbf{k}}(\Theta)&=\prod_{j=1}^N e^{ik_j(\theta_j+\alpha)}\\
    &=e^{i\alpha\sum_jk_j}\phi_{\mathbf{k}}(\Theta).
\end{align}
Thus, after reordering the basis to achieve the decomposition $V^{N,1}_K=B^{N,1}_K\oplus (B^{N,1}_K)^{\perp}$, $\mathcal{D}(Q_{\alpha})$ has block diagonal form:
\begin{equation}
    \mathcal{D}(Q_{\alpha})=\begin{pmatrix}
        \mathrm{Id}& 0\\
        0 & \mathrm{D}_{\alpha, \mathrm{N}}
    \end{pmatrix},
    \label{eq:DQd=1}
\end{equation}
where $\mathrm{D}_{\alpha, \mathrm{N}}$ is a diagonal matrix with entries $e^{i\alpha\sum_jk_j}$ on the row corresponding to $\varphi_{\mathbf{k}}$.

For notational simplicity we introduce $\Dt := \mathcal{D}(Q_t)$, then  $\nabla_{\beta}L_{\mathrm{aug}}(\beta)$ can be readily expressed as 
\begin{align}
    \nabla_{\beta}L_{\mathrm{aug}}(\beta)
    &=\sum_t w_t\Dt^*(A^*A\Dt\beta-A^*Y),
\end{align}
%
%
resulting in the normal equations 
\begin{align}
    \label{eq:normaleqns}
    \bigg(\sum_t w_t\Dt^* A^*A\Dt \bigg) \beta_{\rm aug} =\sum_t w_t\Dt A^*Y.
\end{align}

To understand the effect of data augmentation we need to analyze \eqref{eq:normaleqns}, a task make non-trivial by the fact that the $\Dt$ matrices occur nonlinearly. We begin by decomposing $V^{N,d}_K=B^{N,d}_K\oplus (B^{N,d}_K)^{\perp}$ as indicated above for $d = 1$. We can select a new orthonormal basis that is compatible with this decomposition. With abuse of notation, this results in new design matrices $A$, parameters $\beta$ and $D$-matrices $\Dt$ that have block structure with ${\rm I, N}$, respectively, denoting the invariant and non-invariant parts: 
\begin{equation}
    A=\begin{pmatrix}
        \Ai&\An
    \end{pmatrix},\; 
    A^*=\begin{pmatrix}
            \Ai^*\\
            \An^*
    \end{pmatrix},\; 
    \Dt=\begin{pmatrix}
           \mathrm{Id}&0\\
            0& \Dtn
    \end{pmatrix},\; 
    \beta = \begin{pmatrix}
            \betai\\ \betan
    \end{pmatrix}.
\end{equation}
In this decomposition, the shape of $\Dt$ follows as rotations act trivially on $B^{N,d}_K$ which implies that $(B^{N,d}_K)^{\perp}$ must also be closed under this action. Indeed, acting by a rotation and its inverse $QQ^{-1}$ should be the trivial action.
\begin{remark}
    We can see that
    \begin{equation}
        \esym(\beta)=||\betan||_2 .
    \end{equation}
    \label{remark:esymdec}
\end{remark}
For a given $t$, we obtain 
\begin{align}
    \Dt^* A^*A\Dt &= \begin{pmatrix}
        \Ai^*\Ai&\Ai^*\An \Dtn\label{Eq:Augdm}\\
        \Dtn^*\An^*\Ai&\Dtn^*\An^*\An \Dtn
    \end{pmatrix}.
\end{align}
Lastly, using the fact that the weights of a quadrature rule sum to the measure of the integration space we obtain
\begin{equation}
    \sum_t w_t\Dt^* A^*A\Dt = \begin{pmatrix}
        \Ai^*\Ai& \sum_i w_t \Ai^*\An \Dtn\\
        \sum w_t\Dtn^*\An^*\Ai&\sum_i w_t\Dtn^*\An^*\An \Dtn
    \end{pmatrix}.
    \label{eq:EzAug}
\end{equation}

In the following computations, we will write \eqref{eq:EzAug} in block form
\begin{equation}
    \begin{pmatrix}
        B&C\\
        C^* & D
    \end{pmatrix},
\end{equation}
and assume that the original, un-augmented, design matrix $A$ has full column rank. In our numerical applications, $A$ being of full rank is equivalent to having more data points than parameters.
\begin{lemma}
    If $A$ has full rank then so do $\Ai$ and $\An$.
    \label{lemma:3}
\end{lemma}

This yields the existence of constants $c_0,c_1>0$ such that for any vector $u \in \mathbb{R}^{n_{\rm N}}$ ($n_{\mathrm{N}}$ being the number of columns of $\An$) and any rotation $Q$,
\begin{align}
    c_0\norm{u}&\leq \norm{\An u}\leq c_1\norm{u}  \qquad \text{and} \\
    c_0\norm{u}&\leq \norm{\An D(Q)u}\leq c_1\norm{u},
    \label{eq:LAbounds}
\end{align}
where the second line is a consequence of the unitarity of $D(Q)$. 

%

\begin{lemma}
    \label{lemma:SchurComplement}
    Let $\bar\betai$ denote the solution of the invariant LSQ problem \eqref{eq:invlsq}, then 
    \begin{equation}
        S \beta_{\rm N}
        = 
        \bar{D}_T^* A_{\rm N}^* \big(Y - A_{\rm I} \bar\betai\big),
    \end{equation}
    where $S = D-C^*B^{-1}C$ is the Schur complement \citep{zhangschur} of the block $\Ai^*\Ai$ \eqref{eq:EzAug} and $\overline{D}_T = \sum_t w_t \Dtn$.
    
    $S$ is symmetric, positive definite with spectrum bounded above and below in terms of the singular values of $A$, and therefore, 
    \begin{equation}
        c_2 \| \beta_{\rm N} \| 
        \leq 
        \big\| 
        \bar{D}_T^* A_{\rm N}^* \big(Y - A_{\rm I} \bar\betai\big)
        \big\| 
        \leq 
        c_3 \| \beta_{\rm N} \|,
        \label{eq:SchurSandwich}
    \end{equation}
    where $c_2, c_3$ depend only on the singular values of $A$.
\end{lemma}

\begin{proof}
   We write the of matrix in \eqref{eq:EzAug}. $B$ is invertible as $\Ai$ has full rank (cf. Lemma \ref{lemma:1}). Therefore, its Schur complement
    is invertible and we obtain the following expression for $\betan$
    \begin{align}
        S\betan &= \overline{D}_T^*\An^*Y-C^*B^{-1}C\Ai^*Y\\
        &= \overline{D}_T^*\An(Y-\Ai\Ai^{\dagger}Y).
    \end{align}
    The $\dagger$ symbol denotes the Moore-Penrose pseudoinverse. As $\Ai$ has full rank, $\Ai^{\dagger}Y$ is the solution to the invariant LSQ problem \eqref{eq:invlsq}. Lastly, from $S$'s expression, it is clear that it is positive definite and \eqref{eq:SchurSandwich} follows. 
\end{proof}

Combining Lemma \ref{lemma:SchurComplement} and Remark \ref{remark:esymdec} results in the following upper bound:
\begin{equation}
    \esym(\Pa)\leq \frac{1}{c_2} \norm{\An \overline{D}_T}_{\mathrm{op}}\norm{\Ai \bar\betai-Y},
    \label{eq:SchurUB}
\end{equation}
$\norm{\cdot}_{\mathrm{op}}$ being the operator norm induced by the 2-norm.

\subsubsection{Quadrature Data Augmentation}
From \eqref{eq:EzAug} we observe that if $(w_t,Q_t)_{t=1}^T$ is a quadrature rule for $\mathrm{SO}(d+1)$, then the augmented system becomes block diagonal as quadrature rules integrate polynomials of lower degrees exactly.
\begin{lemma}
    For a LSQ in $V^{N,d}_K$, the entries of the matrix $\mathcal{D}(Q)$ from lemma \ref{lemma:2} are degree $K$ polynomial in $Q$.
    \label{lemma:wignerpoly}
\end{lemma}
This is a standard fact; the $\mathcal{D}$ are finite dimensional representations of $\mathrm{SO}(d+1)$ \citep{goodman2010symmetry}.

\begin{theorem}
    \label{thm:quadaugment}
    Suppose that $(w_t,Q_t)_{t}$ is a quadrature rule with degree of accuracy $n \geq K$, then the augmented LSQ (cf. \eqref{eq:lsqaug}) approximation on $V^{N,d}_K$ is exactly invariant. That is, let $\Pa$ be the resulting polynomial, then
    \begin{equation}
        \esym(\Pa)=0
    \end{equation}
\end{theorem}

\begin{proof}
    From the upper bound obtained in \eqref{eq:SchurUB}, we see that $\overline{D}_T^*=0$ is a necessary condition for $\esym(\Pa)=0$. The conclusion follows from the properties of a quadrature rule and lemma \ref{lemma:wignerpoly}.
\end{proof}

\begin{remark}
    The upper bound in lemma \eqref{eq:SchurUB} we observe that there are no guarantees that $\esym$ improves  before the quadrature rule becomes of sufficiently high degree, it is especially clear in the case of $d=1$ \eqref{eq:DQd=1}. However, it is still possible to treat a insufficient quadrature augmentation as a particular case of a random augmentations, which we study in the next subsection.
    \label{remark:noimprovement}
\end{remark}

\begin{remark}
    The conditions of Theorem \ref{thm:quadaugment} will result in the same approximant as the invariant LSQ problem \eqref{eq:invlsq}. This may be deduced from the fact that the off-diagonal terms in \eqref{eq:EzAug} are integrated exactly to $0$ and that the lower right block is invertible (due to $\An$'s full rank and the unitarity of $\Dtn$).
\end{remark}

In order to benefit from the quadrature augmentation method one has to obtain the appropriate quadrature rules.  For $\mathrm{SO}(2)$, after identifying it with $[0,2\pi)$, we obtain that an equi-weight quadrature rule with degree of accuracy $n-1$ is given by $\big\{\frac{1}{n}, \frac{2\pi t}{n}\colon t\in[n-1]\big\}$ \citep{TrefethenWeideman}. For $\mathrm{SO}(3)$, quadrature rules have been developed in \citep{TUCQuad}, which we use for all our numerical tests. These rules are significantly more computationally expensive both to obtain and to use. 
\begin{table}[]
    \centering
    \begin{tabular}{|c|c|c|}
    \hline
        Degree of accuracy & Number of Points \\
        \hline
        2 & 11\\
        \hline
        4 & 43\\
        \hline
        7 & 168\\
        \hline
        9 & 300\\
        \hline
        14& 960\\
        \hline
        21 &12648 \\
        \hline
        51 & 207576\\
        \hline
    \end{tabular}
    \caption{Size of quadrature rule for $\mathrm{SO}(3)$}
    \label{tab:quadSO3}
\end{table}
We illustrate this with Table \ref{tab:quadSO3} where we provide the size of some quadrature rules \citep{TUCQuad, QuadsGraf}. The rapid increase in the size of quadrature rules cannot be changed.

\begin{lemma}
    For any quadrature rule of degree $n$ on $\mathrm{SO}(3)$ there is a constant $c_1$ such that it contains at least $c_1n^3$ points .
\end{lemma}

\begin{proof}
    Let $V_n$ be the span of the representations of all matrix coefficients with highest weight at most $n$. By the addition theorem of angular momenta, if $f,g\in V_{n/2}$ then $fg\in V_n$.

    Now let $I$ be an integration functional, and $(\mathbf{P},w)$ be a $N$ point quadrature functional, we will name it $J$. Then suppose that they agree on $V_n$. In particular, they also agree on $V_{n/2}$, and since they are linear functionals, they induce an inner product on $V_{n/2}$ and they are of full rank. Moreover, by construction $J$ has rank at most $N$. And $N\geq V_{n/2}> n^3$. 
\end{proof}

\subsubsection{Random Data Augmentation}
The common way of approximating this symmetrisation process is by augmenting via randomly sampled rotations \citep{Augerino, Daoetal}. We show that this leads to a near $T^{-1/2}$ decay in $\esym$.

\begin{theorem}
    Let $A$ be the design matrix. Let $Q_1,Q_2,\cdots$ be iid uniformly drawn elements of $\mathrm{SO}(d+1)$. For fixed $T\in\mathbb{N}$, let $(T^{-1},Q_t)_{t=1}^T$ be augmentation parameters. Let $\Pa$ be the polynomial obtained from the augmented LSQ approximations (cf. \eqref{eq:lsqaug}). Then, for any $\epsilon>0$,
    \begin{equation}
        \frac{\esym(\Pa)\sqrt{T}}{\ln(T)^{1/2+\epsilon}}\xrightarrow[]{a.s.}0
    \end{equation}
    \label{thm:randaugment}
\end{theorem}
\begin{proof}
    We study the random variable
    \begin{equation}
        \overline{D}_T = \frac{1}{T}\sum_t \Dtn.
    \end{equation}
    We know that $\mathbb{E}_{Q_t}(\Dtn)=0$. Thus, for any entry of $\Dtn$ $d_i$, $\mathbb{E}(d_i)=0$ and $\mathbb{E}(d_i^2)<\infty$ as it is a continuous function integrated over a compact set. Therefore, if $\overline{d}_T^i=T^{-1}\sum_t d_t^i$ is an entry of $\overline{D}_T$, according to Theorem 2.5.11 of \citep{Durrett_2019}, for any $\epsilon >0$, 
    \begin{equation}
        \frac{\overline{d}_T\sqrt{T}}{\ln(T)^{1/2+\epsilon}}\xrightarrow[]{a.s.}0.
    \end{equation}
    Moreover, if $\overline{D}_T$ is a square matrix of fixed size $m\times m$, then its Frobenius norm is such that
    \begin{equation}
        ||\overline{D}_T||_F \leq m \max_{kl}|{\overline{d}_T}_{kl}|.
    \end{equation}
    The maximum is run over all entries of $\overline{D}_T$. Additionally, the modulus operation is continuous, thus
    \begin{equation}
        \frac{\norm{\overline{D}_T}_F\sqrt{T}}{\ln(T)^{1/2+\epsilon}}\xrightarrow[]{a.s.}0.
    \end{equation}
    Let $\|\cdot\|$ denote the operator norm induced by the $2$-norm. Since $\|\cdot\|_F\geq \|\cdot\|$, the above convergence property holds for $\|\cdot\|$. Using the submultiplicativity of operator norms we obtain
    \begin{equation}
        \frac{\norm{\An\overline{D}_T^*}\sqrt{T}}{\ln(T)^{1/2+\epsilon}}\xrightarrow[]{a.s.}0.
    \end{equation}
    Applying this result to \eqref{eq:SchurUB}, we obtain the desired result
    \begin{equation*}
        \frac{||\betan||_2\sqrt{T}}{\ln(T)^{1/2+\epsilon}}\xrightarrow[]{a.s.}0. 
        \qedhere 
    \end{equation*}
\end{proof}
\begin{remark}
    Using Theorem 8.5.2 in \citep{Durrett_2019} instead, we may show that there exists $\lambda >0$ such that 
    \begin{equation}
        \limsup_{T\to \infty}\frac{\esym(\Pa)\sqrt{T}}{\sqrt{\ln(\ln(T))}}=\lambda \text{ a.s.}
    \end{equation}
    Therefore, while convergence is faster than $\sqrt{\ln(T)/T}$, this rate is near optimal.
\end{remark}
\begin{remark}
    Nevertheless, if we relax the convergence requirements to only require convergence in probability, we can improve upon the rate of convergence to $T^{-1/2}h(T)$ where $h(T)$ is a measurable function that is unbounded at infinity. Indeed, there exists $\lambda>0$ such that for all $\delta>0$
    \begin{equation}
         \label{eq:improved_MC_rate}
        \mathbb{E}(\esym (\Pa)) \leq \frac{\lambda}{ \sqrt{T}},
    \end{equation}
    and the conclusion follows from Markov's inequality.
\end{remark}
\begin{proof}[Sketch of Proof of \eqref{eq:improved_MC_rate}]
    We use \eqref{eq:SchurUB} and take the expectation to obtain
    \begin{equation}
        c_2\mathbb{E}(\|\beta_N\|_2)\leq \mathbb{E}(\|\overline{D}_T\|)\gamma,
    \end{equation}
    where $\gamma$ is constant (wrt the $Q_t$). 
    Additionally,
    \begin{align}
        \norm*{\frac{1}{T}\sum_t \Dtn}_F^2
        &=\frac{1}{T^2}\left(\|\Dtn\|^2_F+2\sum_{k\neq l}\Tr( D_{k,\mathrm{N}}^{*} D_{l,\mathrm{N}})\right).
    \end{align}
    Taking the expectation, we obtain,
    \begin{align}
        \mathbb{E}(\|\overline{D}_T\|_F^2)&= \frac{1}{T^2}\left(\sum_t\mathbb{E}(\|D_{t,\mathrm{N}}\|_F^2)+2\sum_{k\neq l}\Tr(\mathbb{E}(D_{k,\mathrm{N}}^*D_{l,\mathrm{N}}))\right)\\
        &=\frac{1}{T^2}\left(\sum_t\mathbb{E}(\|D_{t,\mathrm{N}}\|_F^2)+2\sum_{k\neq l}\Tr(\mathbb{E}(D_{k,\mathrm{N}}^*)\mathbb{E}(D_{l,\mathrm{N}}))\right)\\
        &=\frac{1}{T}\mathbb{E}(\|D_{1,\mathrm{N}}\|_F^2)\eqdef \frac{\mu}{T}.
    \end{align}
    With $\mu = \mathbb{E}(\|D_{1,\mathrm{N}}\|^2_F)$. This quantity is well defined as it is a real-valued integral of a continuous function over a compact set.
    Jensen's inequality, the fact that $\|\cdot\|_F\geq \|\cdot\|_2$ and Remark \ref{remark:esymdec} yield the desired result.
\end{proof}

While random augmentations are easy to implement, it has an asymptotic error rate lower bounded by $T^{-1/2}$, which, given the discussion in Section \ref{section:AngMom}, is insufficient.


\section{Numerical Tests}
\label{sec:NumTest}
We illustrate the theory developed in the previous section with a range of numerical tests. For these tests, we choose $d \in \{2, 3\}$ and $N=3$, i.e. systems consisting of three two- or three-dimensional particles. The choice $N = 3$ corresponds to the simplest non-trivial invariant basis. Indeed, for $N=1$, the only rotation-invariant functions are the constant functions, while for $N=2$, the position of any one particle determines the entire system. Only from $N=3$ onwards do we observe non-trivial couplings. 

Restricting our tests to relatively few particles is also motivated by the fact that in much larger point cloud models, dimension reduction techniques are always applied to reduce the many particle interaction to moderate correlation order. For example, the highly successful MACE model~\citep{MACE} employs just two layers of $N=3$ for modelling complex interatomic interactions. 

An additional common choice in molecular dynamics especially for materials systems is that particles are rotated into a reference frame (natural unit cells or supercells). We therefore also consider non uniform distributions. However, these distributions have the property that when they are averaged over $\mathrm{SO}(d+1)$ they become the uniform distribution. 
Except when otherwise stated, we only vary the distributions when fitting the model, the error will be computed using the uniform distribution. Our models are not permutation invariant, thus when sampling the non-uniform distributions, the choice of the constrained particles is fixed throughout the tests. It would have been possible to make our bases permutation invariant, but we chose not to as we wish to focus on the specific case of rotation invariance.

\subsection{Results in Two Dimensions}

\begin{table}[!h]
    \centering
    \begin{tabular}{|c|c|c|c|}
        \hline
                       & Particle 1 & Particle 2 & Particle 3 \\
                       \hline
        $UUU$ &  $U$ &  $U$ &  $U$ \\
        \hline
        $\delta UU$  & $\delta$ &  $U$ &  $U$ \\
        \hline
        $\delta^*UU$ & $M$ &  $U$ &  $U$ \\

        \hline

    \end{tabular}
    \caption{The used measures are the product of all measures in a row. $U$ is the uniform measure, $\delta$ the Dirac measure at 0 and $\delta^*$ is a von Mieses measure concentrated at 0. In this paper, we chose $\kappa = 100$. It is intended to be a mollification of the dirac measure.}
    \label{tab:1}
\end{table}

\begin{figure}
    \centering
    \begin{subfigure}{0.3\textwidth}
        \centering
        \includegraphics[width=\textwidth]{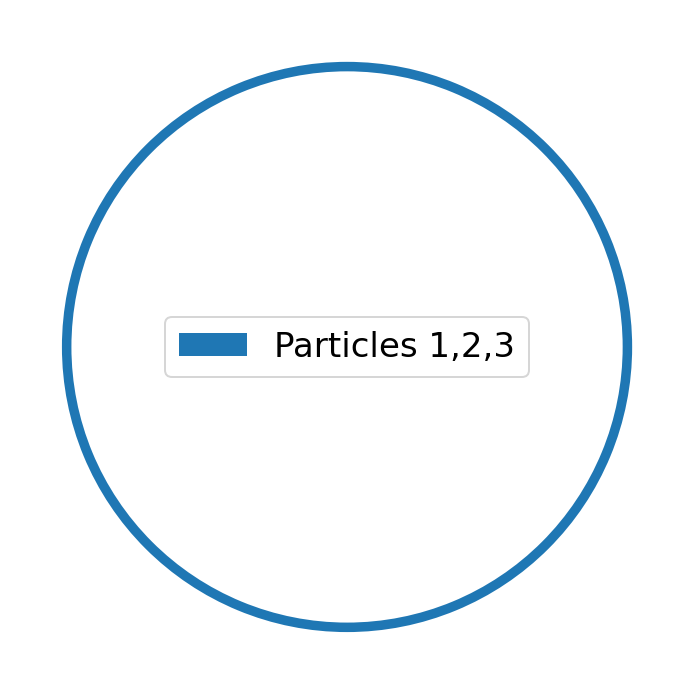} 
        \caption{$UUU$}
        \label{fig:2dD1}
    \end{subfigure}
    \begin{subfigure}{0.3\textwidth}
        \centering
        \includegraphics[width=\textwidth]{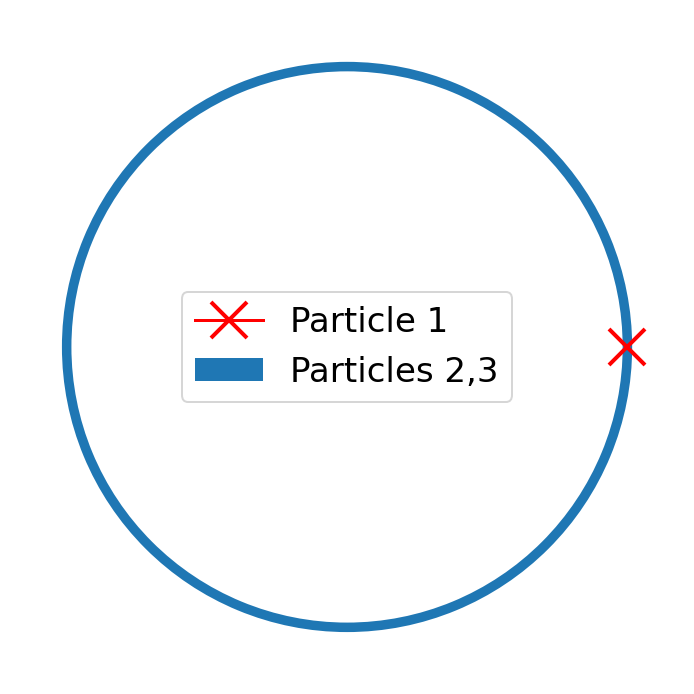} 
        \caption{$\delta UU$}
        \label{fig:2dD2}
    \end{subfigure}
    \begin{subfigure}{0.3\textwidth}
        \centering
        \includegraphics[width=\textwidth]{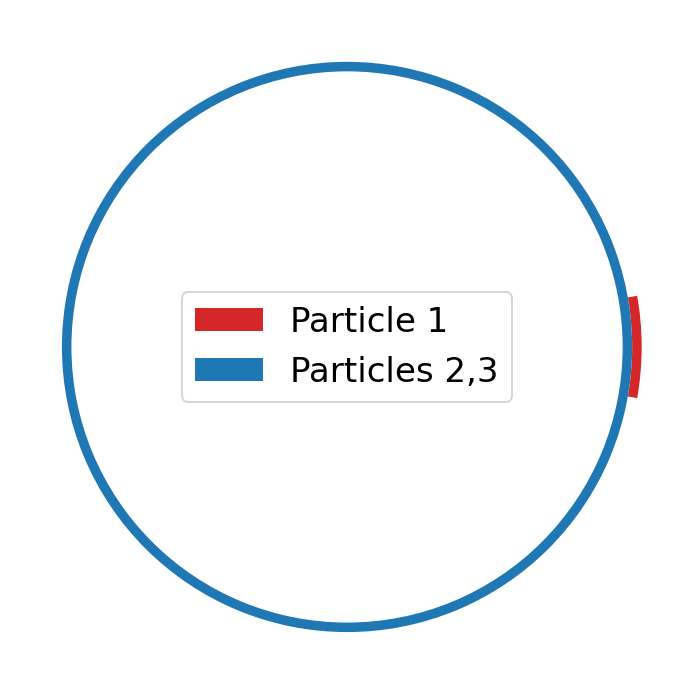} 
        \caption{$\delta^*UU$}
        \label{fig:2dD3}
    \end{subfigure}
    \caption{An illustration of the distributions mentioned in Table \ref{tab:1}}.
    \label{fig:distributions2D}
\end{figure}

The distributions used for two-dimensional tests are summarized in Table \ref{tab:1} and illustrated in Figure \ref{fig:distributions2D}.

As target function we choose a high-degree polynomial $f \in B_{30}^{3,1}$, i.e.,  such that its coefficients  have an exponential decay of rate $\alpha=2$, 
\begin{equation}
    f(\Theta) = \sum_{\mathbf{k}\in\mathbb{Z}^N,\norm{\mathbf{k}}_{1}\leq \tilde{K}}c_\mathbf{k} e^{-\alpha||\mathbf{k}||_1} e^{\mathbf{k}\cdot \Theta}, 
\end{equation}
with $c_\mathbf{k}$ drawn uniformly from $[-1,1]$. The exponential decay is a natural generalization of analyticity in the univariate setting.
Due to the assumed decay, it is straightforward to see that there exists $C>0$ such that  
\begin{equation}
    ||f-f_K||_{L_2} \leq Ce^{-\alpha K},
\end{equation}
where $f_K \in B^{3,1}_K$ is the degree $K$ truncation ($L^2$-orthogonal projection) of $f$.
For sufficient training, all LSQ methods should attain this approximation error rate.

\begin{figure}
    \centering
    \begin{subfigure}{0.35\textwidth}
        \centering
        \includegraphics[height=3.8cm]{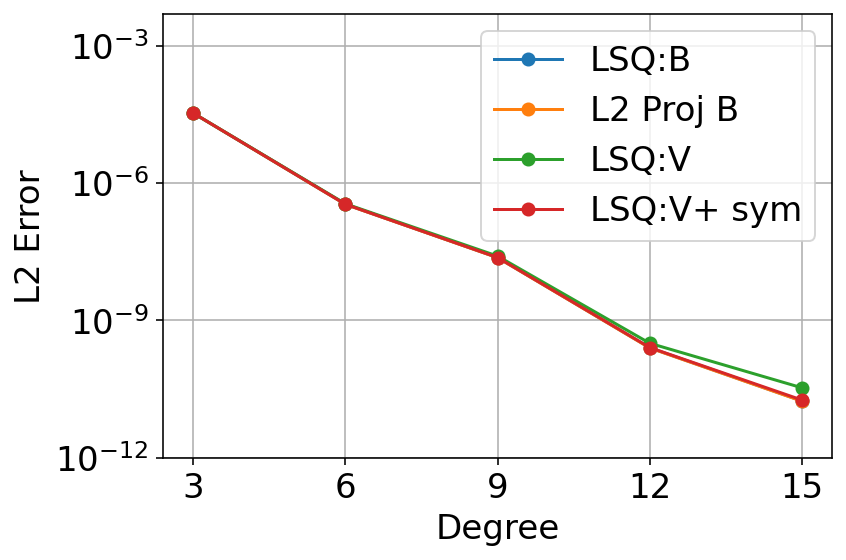} 
        \caption{$UUU$}
        \label{fig:LSQ2dD1}
    \end{subfigure}
    \begin{subfigure}{0.3\textwidth}
        \centering
        \includegraphics[height=3.7cm]{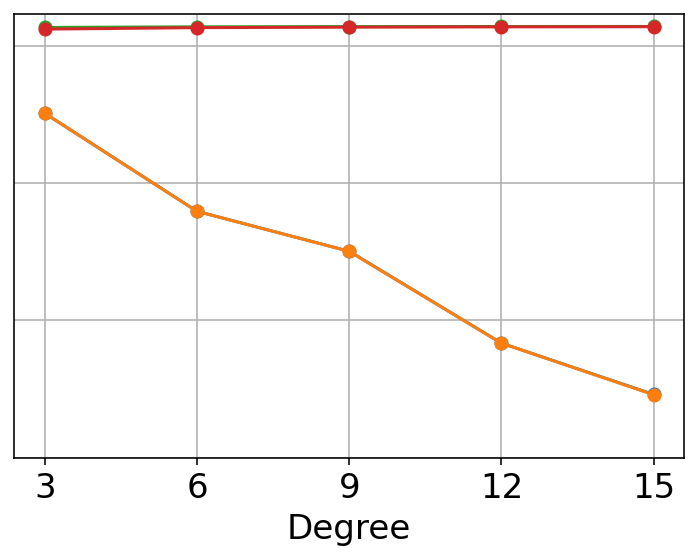} 
        \caption{$\delta UU$}
        \label{fig:LSQ2dD2}
    \end{subfigure}
    \begin{subfigure}{0.3\textwidth}
        \centering
        \includegraphics[height=3.7cm]{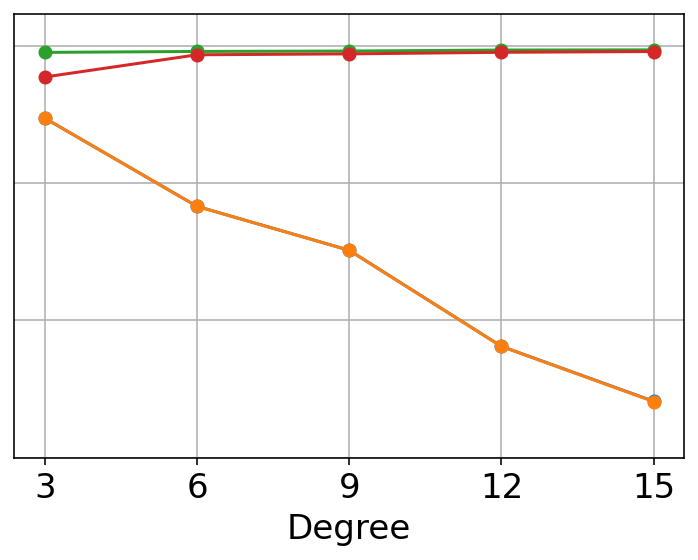} 
        \caption{$\delta^*UU$}
        \label{fig:LSQ2dD3}
    \end{subfigure}
    \caption{Approximation rates for the different distributions. Our sample size is 8000 and testset size is 2000. The design matrix for $\delta^*UU$ was found to be ill-conditioned, thus the LSQ required regularisation, which we performed by choosing a cutoff singular value of $10^{-4.5}$. This cutoff has been chosen as the one that minimises the approximation error. All three graphs share the same the y-axis. Remark that the blue, green and red plots are indistinguishable.}
    \label{fig:LSQ2d}
\end{figure}

Figure \ref{fig:LSQ2d} shows our LSQ scheme's performances when sampling the different distributions (cf Table \ref{tab:1}); the rotation invariant LSQ \eqref{eq:invlsq} always achieves the optimal approximation error. This is also the case for the full basis LSQ \eqref{eq:fullLSQ} with $UUU$ samples (see subfigure \ref{fig:LSQ2dD1}). This approximation error is not exactly optimal for the higher degrees, as we are only using a training set of size 8000. However, we can see that $\sym$ improves this error to the optimal one. Subfigure \ref{fig:LSQ2dD2} shows that fitting the full basis is impossible given data sampled from $\delta UU$ as it is lacking information on Particle 1. Figure \ref{fig:LSQ2dD3} shows that this is also the case when the constraint on Particle $1$ has slightly been relaxed.

\begin{figure}
    \centering
    \begin{subfigure}{0.35\textwidth}
        \centering
        \includegraphics[height=3.8cm]{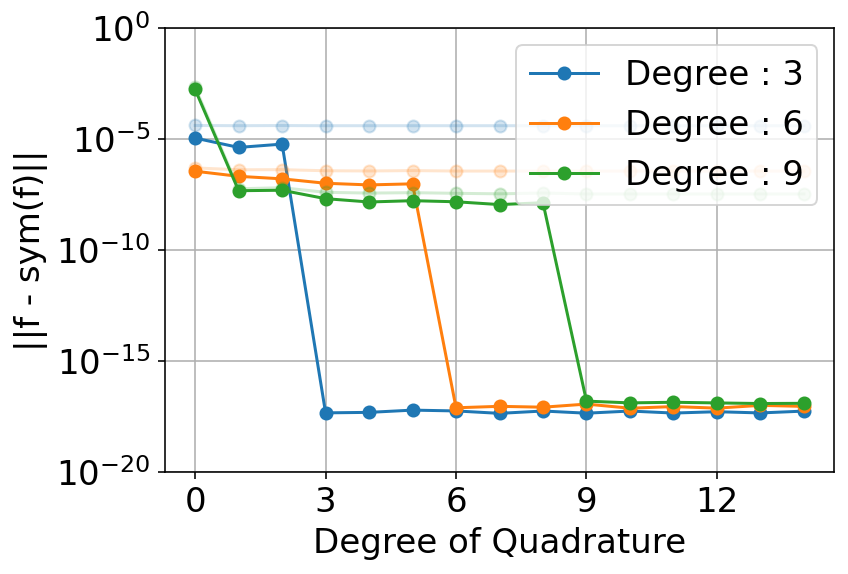} 
        \caption{$UUU$}
        \label{fig:quadpsym2dD1100D1}
    \end{subfigure}
    \begin{subfigure}{0.3\textwidth}
        \centering
        \includegraphics[height=3.7cm]{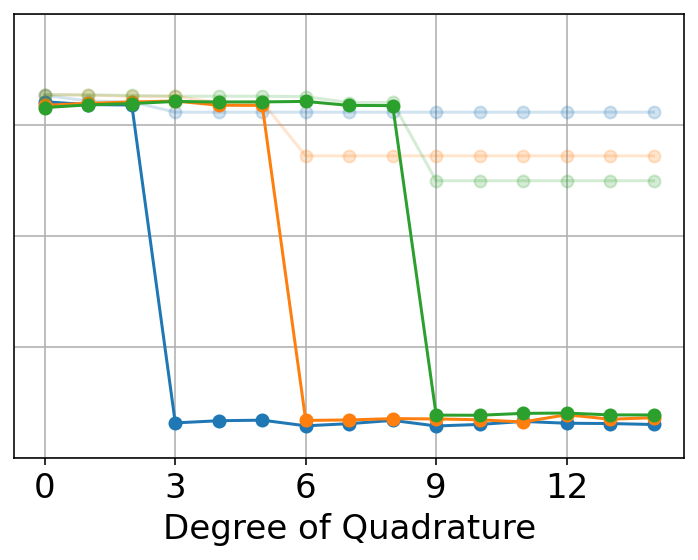} 
        \caption{$\delta UU$}
        \label{fig:quadpsym2dD2100D2}
    \end{subfigure}
    \begin{subfigure}{0.3\textwidth}
        \centering
        \includegraphics[height=3.7cm]{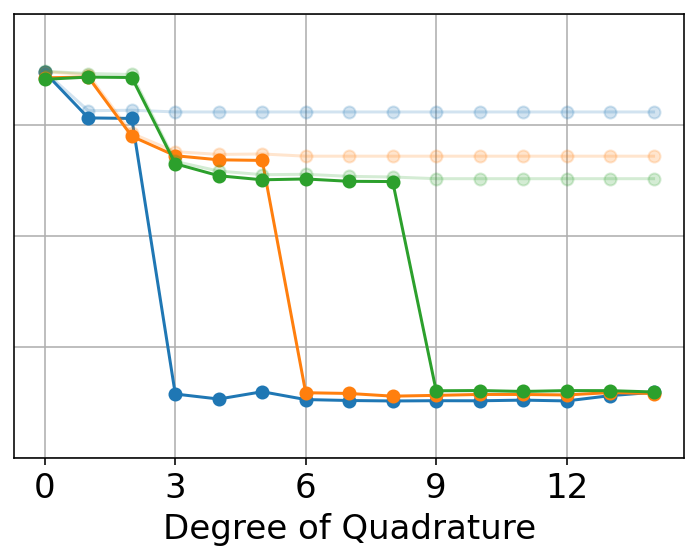} 
        \caption{$\delta^*UU$}
        \label{fig:quadpsym2dD3100D3}
    \end{subfigure}
    \caption{$\varepsilon_{sym}$ for quadrature augmentation using $UUU$, $\delta UU$, $\delta^*UU$. The (unaugmented) training set size is 800 and the validation set size is 200. All figures share the same y-axis. The shaded lines correspond to the approximation error.}
    \label{fig:quadpsym2d100}
\end{figure}

On Figure \ref{fig:quadpsym2d100} we can see that a quadrature augmented LSQ \eqref{eq:lsqaug} achieves exact symmetry learning for all sampling distributions. In fact, exact symmetry is achieved immediately after the quadrature rule's degree overtakes that of the LSQ system. We observe that this numerical experiment also illustrates Remark \ref{remark:noimprovement} as there is no meaningful improvements in $\esym$ before the quadrature degree overcomes the approximation degree. If a quadrature augmentation symmetrised all the lower degree terms of the LSQ approximation, then we would expect an exponential decay in the symmetrisation error due to the decay in the polynomial coefficients - this follows from Remark \ref{remark:esymdec}.

\begin{figure}
    \centering
    \begin{subfigure}{0.35\textwidth}
        \centering
        \includegraphics[height=3.7cm]{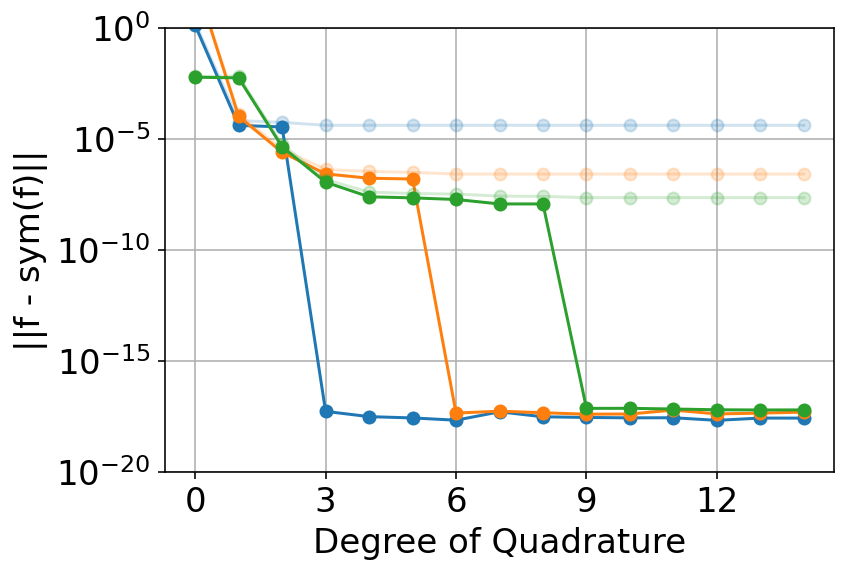} 
        \caption{$d=1$}
        \label{fig:quadpsym2dD1100}
    \end{subfigure}
    \caption{$\esym$ when no regularisation is applied for quadrature augmentation using $\delta^*UU$. The training set-size is 800 and the validation set-size is 200. }
    \label{fig:NRquad2d}
\end{figure}

In the proofs of Theorem \ref{thm:quadaugment} and $\ref{thm:randaugment}$ we supposed that the design matrix had full rank for the sake of simplicity. However, this is not the case for the full basis LSQ \eqref{eq:fullLSQ} when sampling $\delta UU$. And when sampling $\delta^*UU$, while the resulting design matrix has full rank, it is ill-conditioned. For instance, for for an approximation of degree $9$ and $800$ data points in the learning set, the smallest singular value is on the order of $10^{-6}$. This explains the need to regularise. However, data augmentation leads to better conditioned systems and quadrature augmentation will still lead to exact symmetry learning. This is illustrated in Figure \ref{fig:NRquad2d}.

\begin{figure}[!htb]
    \centering
    \begin{subfigure}{0.35\textwidth}
        \centering
        \includegraphics[height=3.8cm]{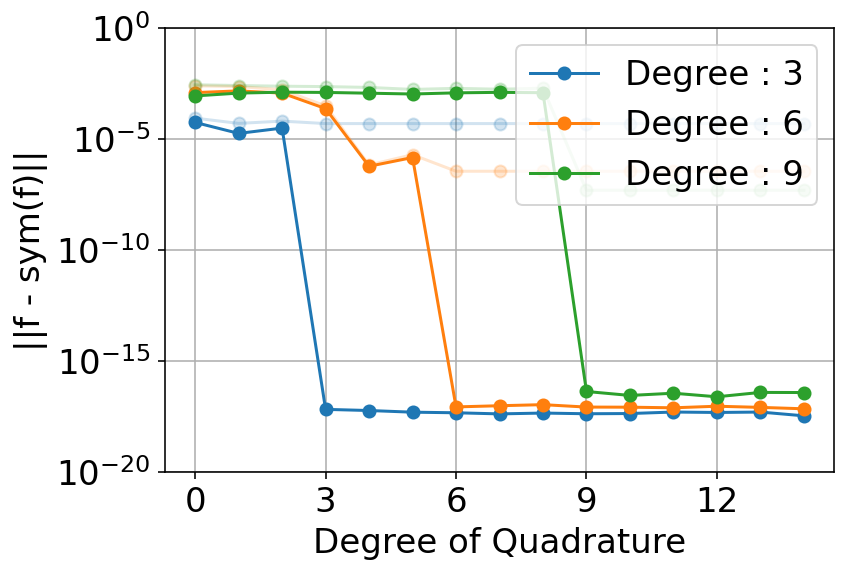} 
        \caption{$UUU$}
        \label{fig:quadpsym2dD1low}
    \end{subfigure}
    \begin{subfigure}{0.3\textwidth}
        \centering
        \includegraphics[height=3.7cm]{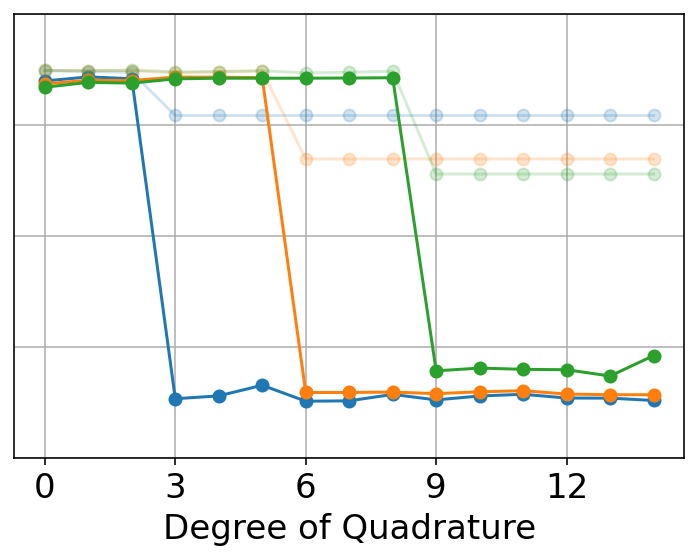} 
        \caption{$\delta UU$}
        \label{fig:quadpsym2dD2low}
    \end{subfigure}
    \begin{subfigure}{0.3\textwidth}
        \centering
        \includegraphics[height=3.7cm]{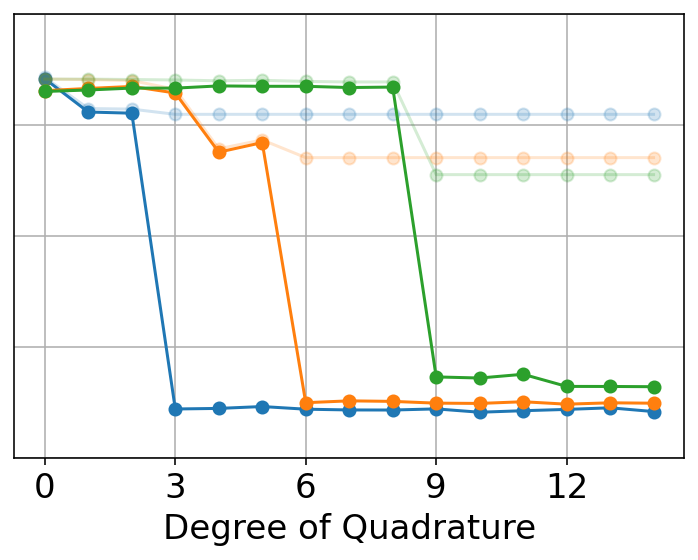} 
        \caption{$\delta^*UU$}
        \label{fig:quadpsym2dD3low}
    \end{subfigure}
    \caption{$\varepsilon_{sym}$ for quadrature augmentation for the $UUU$, $\delta UU$ and $\delta^*UU$. The (unaugmented) training data set size is now 100. The validation set size is 100. All figures share the same y-axis. The shaded lines correspond to the approximation error.}
    \label{fig:quadpsym2dlow}
\end{figure}
In Figure \ref{fig:quadpsym2dlow} we have now decreased the training dataset to 100 points. This experiment demonstrates that even for a insufficiently sized dataset, quadrature augmentation will eventually lead to perfect symmetrisation. 

\begin{figure}[!htb]
    \centering
    \begin{subfigure}{0.35\textwidth}
        \centering
        \includegraphics[height=3.8cm]{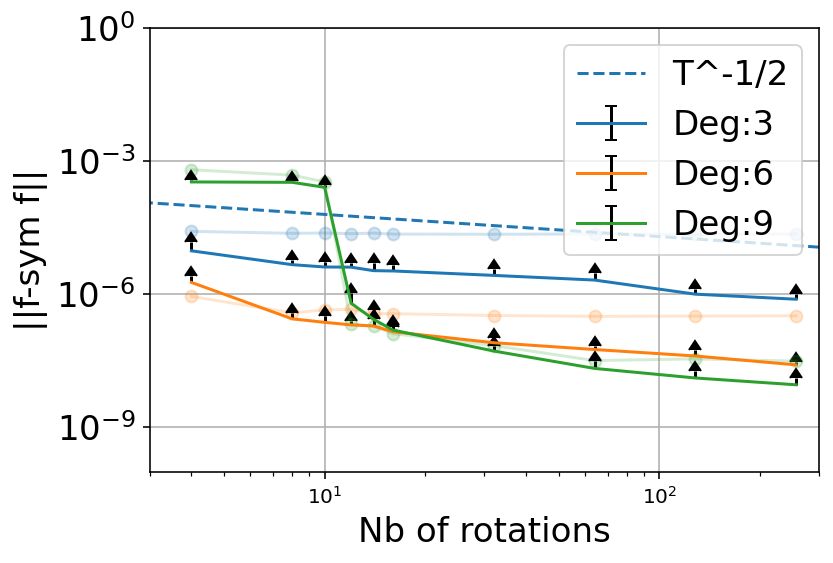} 
        \caption{$UUU$}
        \label{fig:Augpsym2dD1}
    \end{subfigure}
    \begin{subfigure}{0.3\textwidth}
        \centering
        \includegraphics[height=3.7cm]{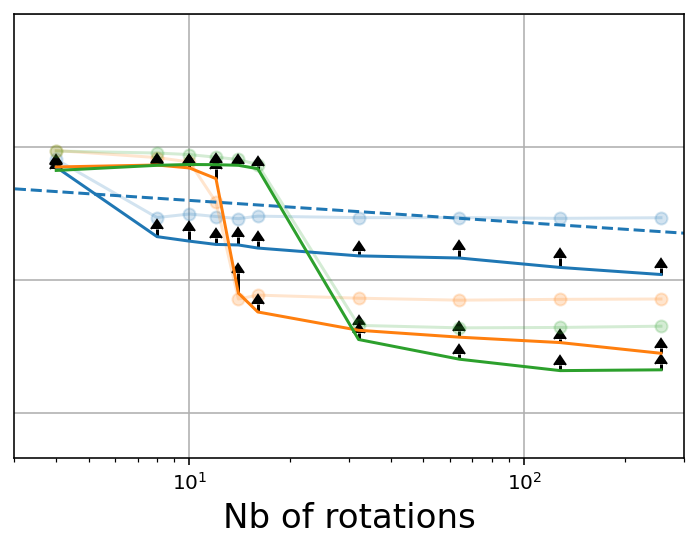} 
        \caption{$\delta UU$}
        \label{fig:Augpsym2dD2}
    \end{subfigure}
    \begin{subfigure}{0.3\textwidth}
        \centering
        \includegraphics[height=3.7cm]{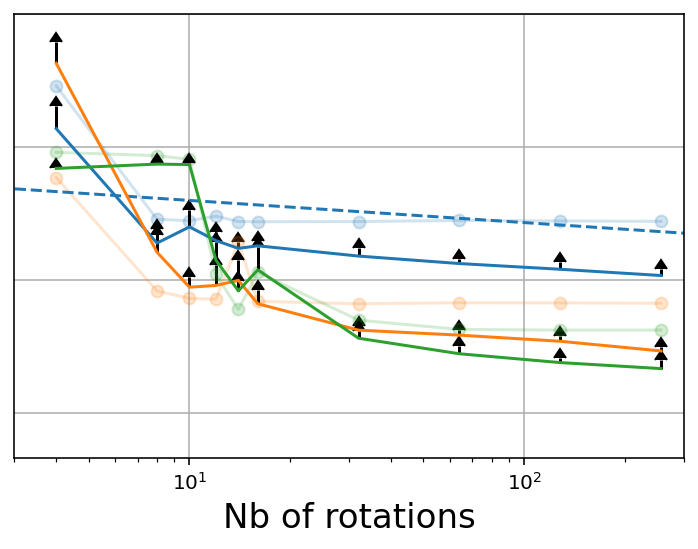} 
        \caption{$\delta^*UU$}
        \label{fig:Augpsym2dD3}
    \end{subfigure}
    \caption{$\varepsilon_{sym}$ for a random augmentation for $UUU$, $\delta UU$ and $\delta^*UU$. The (unaugmented) training data set size is 100 and the validation set size is 100. The number of rotations are powers of 2 between 4 and 256. For each distribution we simulate 10 independent scenarii. The mean results are the plotted lines and the error-arrow represents 1 standard deviation. There is no lower arrow to help with visualisation as around the drops the simulations are unstable and the error interval reaches into the negatives. The dotted blue line represents a $T^{-1/2}$ decay. The shaded lines correspond to the approximation error of the first trial.
    \label{fig:Augpsym2d}}
\end{figure}

Figure \ref{fig:Augpsym2d} shows us the result of random augmentation \eqref{eq:lsqaug}. In the asymptotic regime we clearly observe the rate $T^{-1/2}$ as predicted by \ref{thm:randaugment}. However, in the pre-asymptotic regime we observe a sharp drop in the symmetrization error which we hypothesize corresponds to the moment all parameters are fully determined. This is especially well illustrated by the concentrated distributions. Indeed, some basis elements are undetermined in the unaugmented dataset since particle 1 is always constrained to be at $\theta=0$. This can also be observed in the relatively early drop for $\delta^*UU$ that can be observed in Figure \ref{fig:Augpsym2dD3} as in that case Particle 1 is not fully  constrained and thus these basis elements have relatively more information. Additionally, around that drop the standard deviation is quite high, which can be explained by the fact that the LSQ system is on the cusp of having enough data to fully determine these few basis elements.  However, there is one a priori counterintuitive observation. For a fixed augmentation $\varepsilon_{\sym}$ decreases when the degree increases. One would expect that a lower degree approximation is easier to symmetrise since there are less parameters. The upper bound in \eqref{eq:SchurUB} may explain this phenomenon. Due to the high regularity of our target function, for a well-fitted invariant LSQ \eqref{eq:invlsq}, the term $\norm{Y-\Ai \overline{\betai}}$ decays exponentially in $K$. And so even if $\norm{\overline{D}_T^*\An^*}$ were to grow for increasing $K$, overall, $\norm{\overline{D}_T^*\An^*}\norm{Y-\Ai \overline{\betai}}$ will be decreasing in $K$ for a sufficiently smooth target. We illustrate this behavior in the appendix (section \ref{Appendix}).

\begin{figure}[!htb]
    \centering
    \begin{subfigure}{0.35\textwidth}
        \centering
        \includegraphics[height=3.8cm]{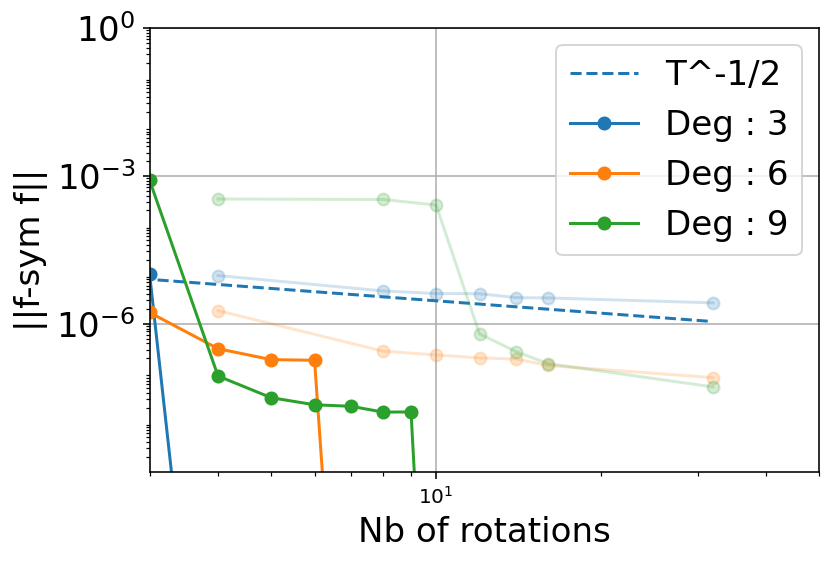} 
        \caption{$UUU$}
        \label{fig:comppsym2dD1}
    \end{subfigure}
    \begin{subfigure}{0.3\textwidth}
        \centering
        \includegraphics[height=3.7cm]{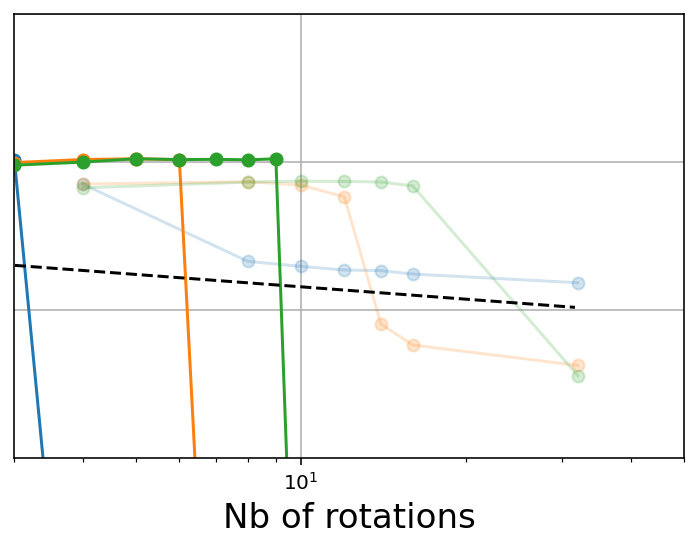} 
        \caption{$\delta UU$}
        \label{fig:comppsym2dD2}
    \end{subfigure}
    \begin{subfigure}{0.3\textwidth}
        \centering
        \includegraphics[height=3.7cm]{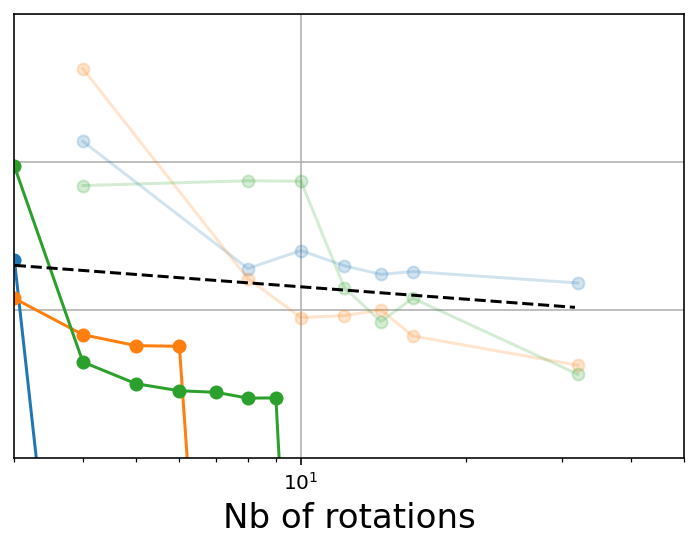} 
        \caption{$\delta^*UU$}
        \label{fig:comppsym2dD3}
    \end{subfigure}
    \caption{The superposition of Figures \ref{fig:quadpsym2d100} and \ref{fig:Augpsym2d}. The shaded plots correspond to random augmentation while the solid ones correspond to quadrature augmentation. All plots share the same y-axis.
    \label{fig:compare2d}}
\end{figure}

Figure \ref{fig:compare2d} compares the symmetrisation error $\esym$ of both random and quadrature augmentation \eqref{eq:lsqaug}. While our proofs do not guarantee that quadrature augmentation will be significantly better than random augmentation before exact symmetry learning (cf. Remark \ref{remark:noimprovement}), we observe that empirically quadrature augmentation is an order of magnitude better than random augmentations. Remark that the square root behaviour is not visible as the system has not yet reached the asymptotic regime.

\subsection{Results in Three Dimensions}

\begin{table}[!h]
    \centering
    \begin{tabular}{|c|c|c|c|}
        \hline
                       & Particle 1 & Particle 2 & Particle 3 \\
                       \hline
        $UUU$ &  $U$ & $U$ & $U$ \\
        \hline
        $\delta UU$  & $\delta$ & $U$ & $U$ \\
        \hline
        $\delta^*UU$ & $\delta^*$ & $U$ & $U$ \\
        \hline
        $\delta H_1U$ & $\delta$ & $H_1$ & $U$\\
        \hline
        $\delta^*H_1^*U$ & $\delta^*$ & $H_1^*$ & $U$\\
        \hline

    \end{tabular}
    \caption{The used measures are the product of all measures in a row. $U$ is the uniform measure on $S^2$, $\delta$ the Dirac measure at the north pole, $H_1$ is a the uniform measure on a geodesic passing through the north pole. The symbol $*$ denotes a mollification of the measure, i.e., additional Gaussian noise. For our experiments we chose to perturb the spherical angles independently by a Gaussian with a standard deviation of 0.1. }
    \label{tab:2}
\end{table}

\begin{figure}
    \centering
    \begin{subfigure}{0.3\textwidth}
        \centering
        \includegraphics[height=3.7cm]{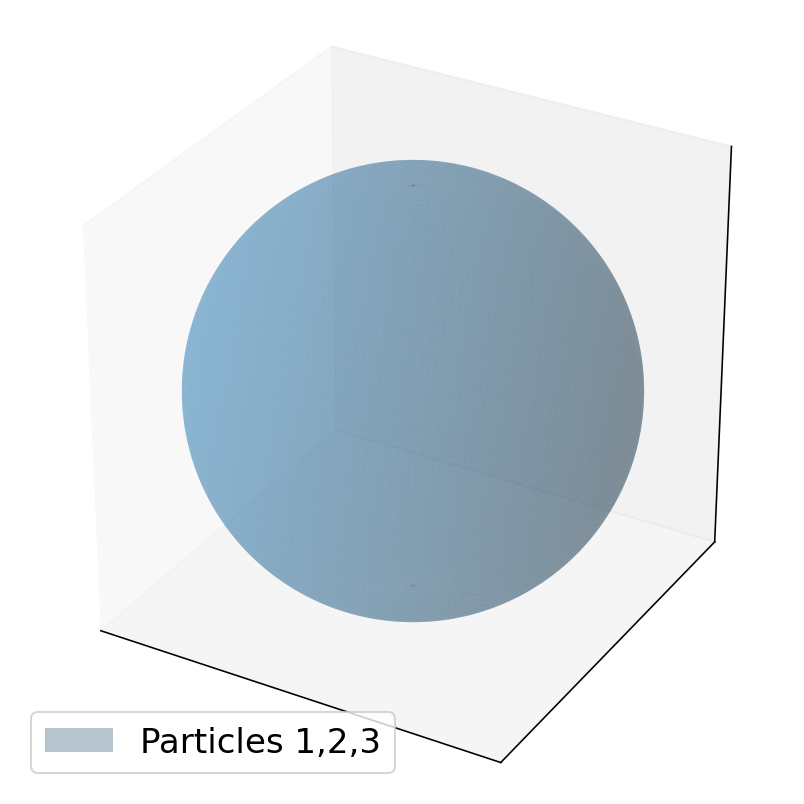} 
        \caption{$UUU$}
        \label{fig:3dD1}
    \end{subfigure}
    \begin{subfigure}{0.3\textwidth}
        \centering
        \includegraphics[height=3.7cm]{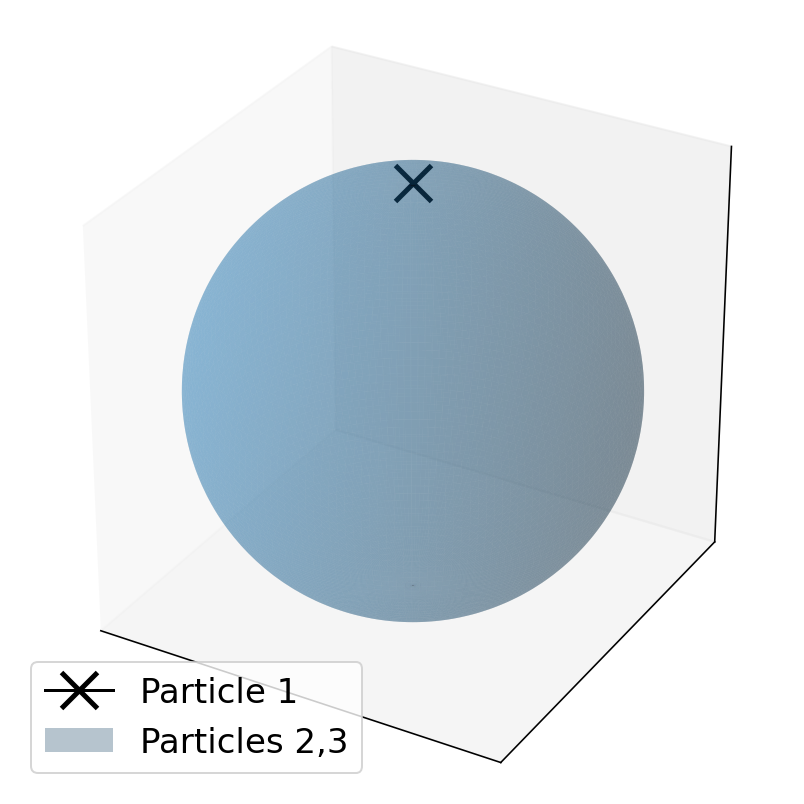} 
        \caption{$\delta UU$}
        \label{fig:3dD2}
    \end{subfigure}
    \begin{subfigure}{0.3\textwidth}
        \centering
        \includegraphics[height=3.7cm]{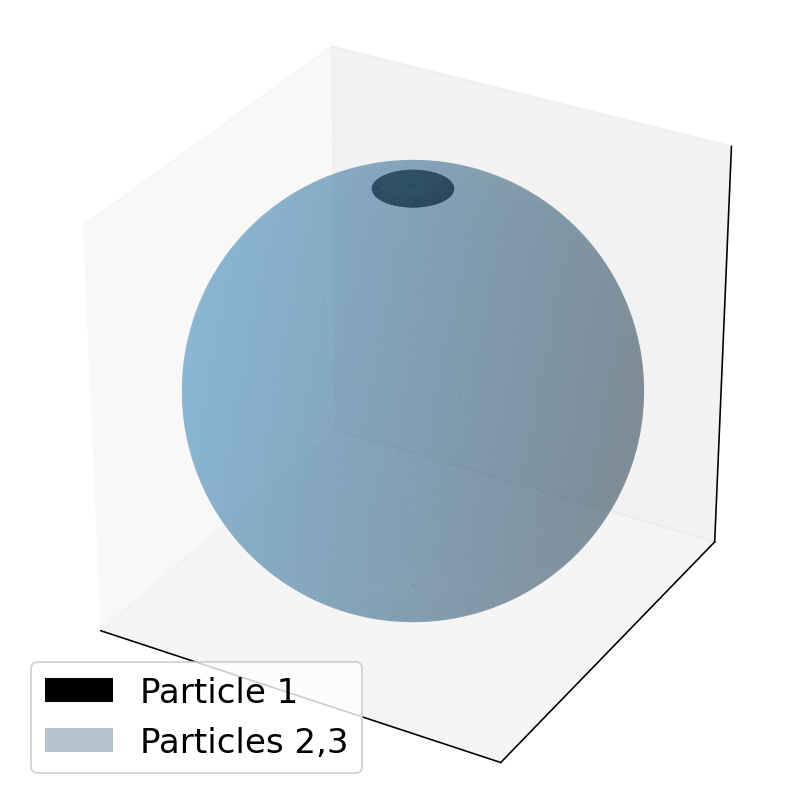} 
        \caption{$\delta^* UU$}
        \label{fig:3dD3}
    \end{subfigure}
    \vspace{1em}
    \begin{subfigure}{0.3\textwidth}
        \centering
        \includegraphics[height=3.7cm]{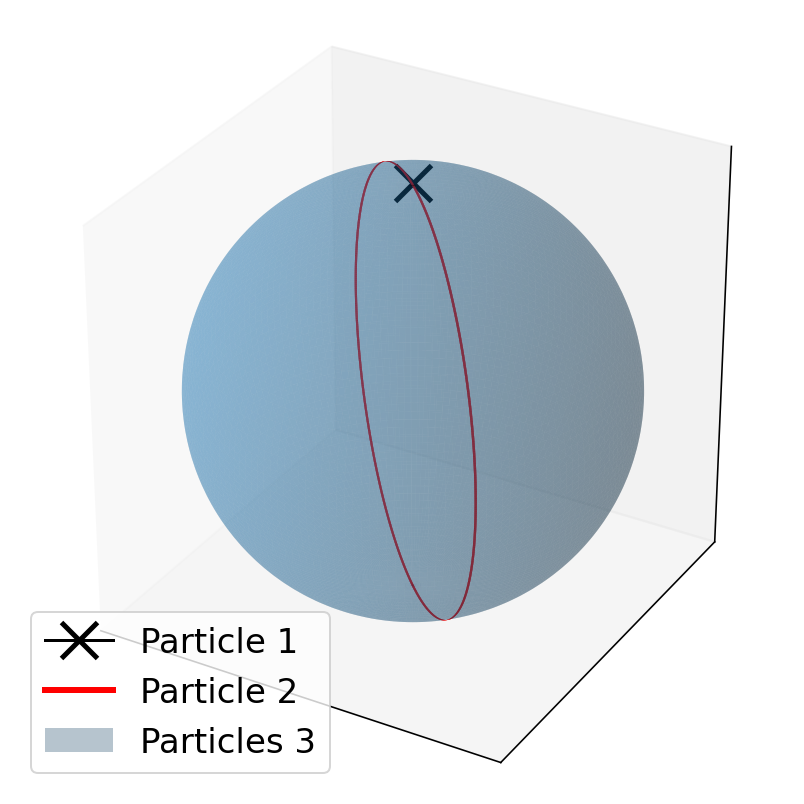} 
        \caption{$\delta H^1U$}
        \label{fig:3dD4}
    \end{subfigure}
    \begin{subfigure}{0.3\textwidth}
        \centering
        \includegraphics[height=3.7cm]{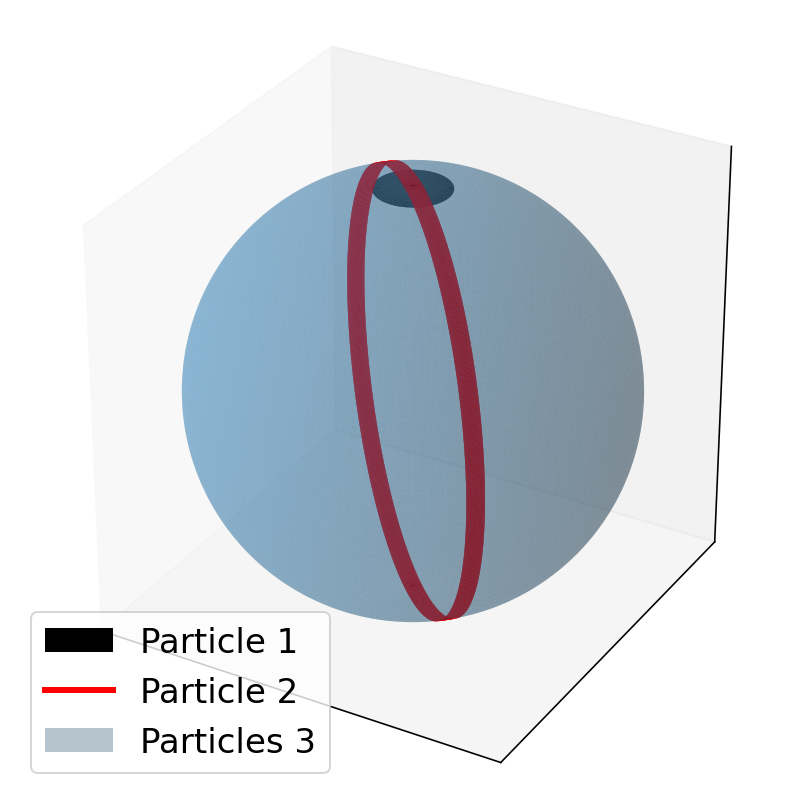} 
        \caption{$\delta^* H_1^*U$}
        \label{fig:3dD5}
    \end{subfigure}
    \caption{Illustrations of the distributions mentioned in Table \ref{tab:2}}.
    \label{fig:distributions3D}
\end{figure}

The data distributions used in the three-dimensional experiments are summarized in Table \ref{tab:2} and illustrated in Figure \ref{fig:distributions3D}.

As we work with 3 particles, we may use \eqref{eq:3dIbasis} to build our approximation schemes as it enumerates an orthonormal basis for the space $B^{3,2}_{K}$. 
Again, we choose our target polynomial to have an exponential decay of rate $\alpha=2$ in its coefficients in order to mimic analyticity. That is, we may write our target function $f$ in the following form:
\begin{equation}
    f = \sum_{\mathbf{l}}c_{\mathbf{\mathbf{l}}}e^{-\alpha \norm{\mathbf{l}}_1}\tilde{\varphi}_{\mathbf{l}},
\end{equation}
where the $c_\mathbf{l}$ are chosen uniformly on $[-1,1]$.
Throughout our tests we observe qualitatively very similar results as in our two-dimensional tests. 

\begin{figure}
    \centering
    \begin{subfigure}{0.35\textwidth}
        \centering
        \includegraphics[height=3.8cm]{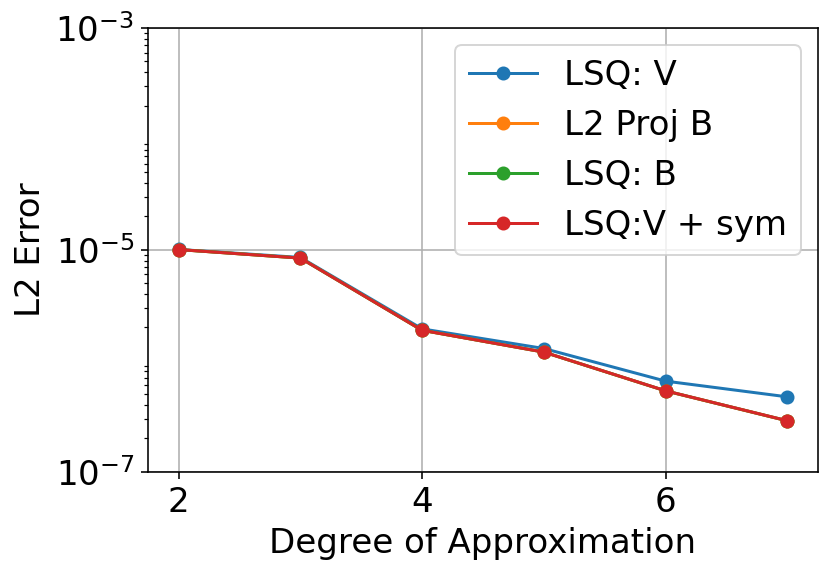} 
        \caption{$UUU$}
        \label{fig:LSQ3DD1}
    \end{subfigure}
    \begin{subfigure}{0.3\textwidth}
        \centering
        \includegraphics[height=3.7cm]{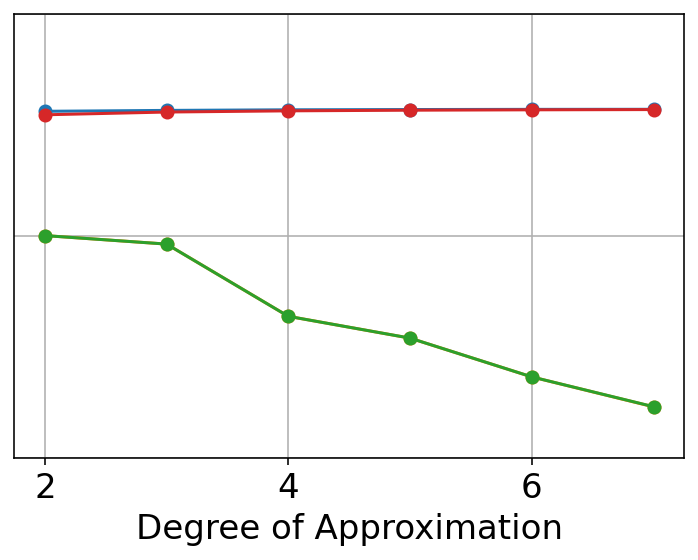} 
        \caption{$\delta UU$}
        \label{fig:LSQ3DD2}
    \end{subfigure}
    \begin{subfigure}{0.3\textwidth}
        \centering
        \includegraphics[height=3.7cm]{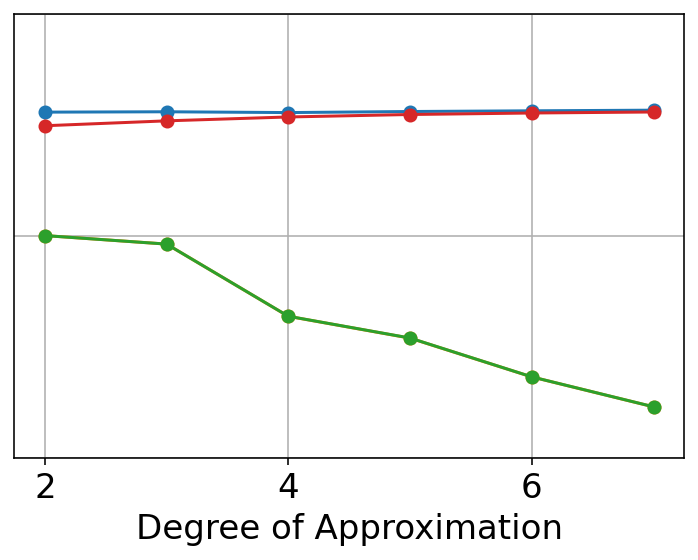} 
        \caption{$\delta^*UU$}
        \label{fig:LSQ3DD3}
    \end{subfigure}
    \vspace{1em}
        \begin{subfigure}{0.3\textwidth}
        \centering
        \includegraphics[height=3.7cm]{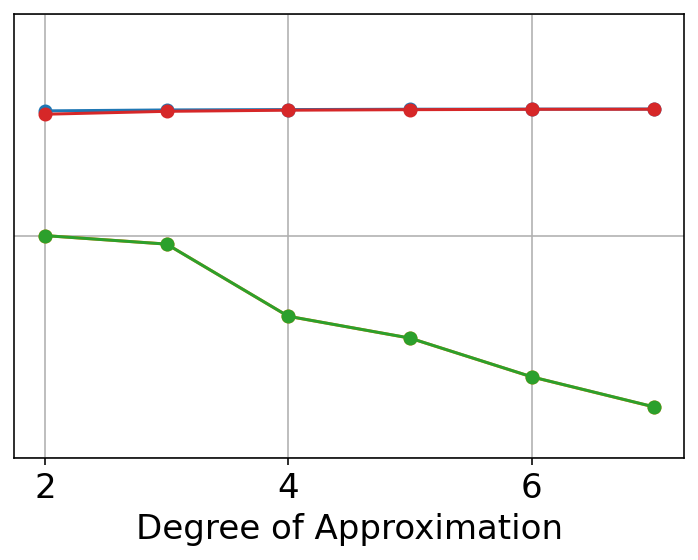} 
        \caption{$\delta H_1U$}
        \label{fig:LSQ3DD4}
    \end{subfigure}
    \begin{subfigure}{0.3\textwidth}
        \centering
        \includegraphics[height=3.7cm]{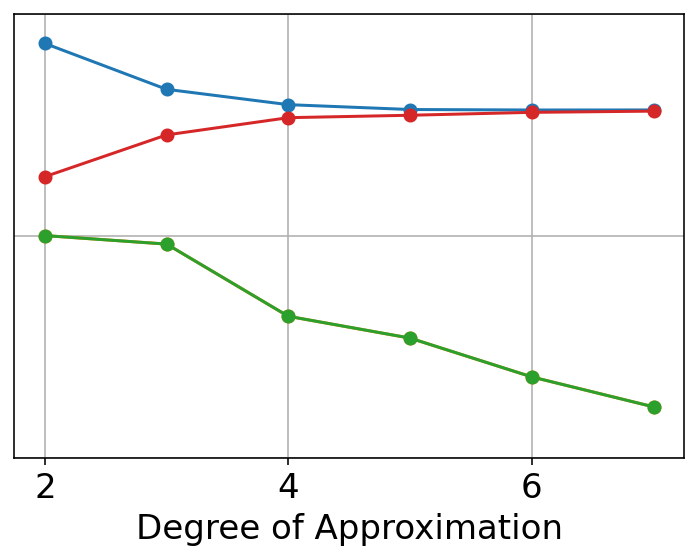} 
        \caption{$\delta^*H_1^*U$}
        \label{fig:LSQ3DD5}
    \end{subfigure}
    
    \caption{Approximation rates for the distributions in Table \ref{tab:2}. We approximate a degree 11 rotation-invariant function with 10000 points in the training set and 2500 points in the validation set. As before, the perturbed distributions $\delta^*UU$ and $\delta^*H_1^*U$ need to be regularised. All plots share the same y-axis.}
    \label{fig:LSQ3D}
\end{figure}
In Figure \ref{fig:LSQ3D} we recover the spectral convergence rate of the optimal polynomial approximation of an analytic function. Again, we had to use some regularisation for $\delta^*UU$ and $\delta^*H_1^*U$. The optimal cutoffs of singular values were found to be, respectively, $10^{-3.4}$ and $10^{-5}$.

Next, to test quadrature augmentation, we need to obtain a quadrature rule for the group average over $\mathrm{SO}(3)$. In general, obtaining quadrature rules on manifolds, for instance Lie Groups, is a non-trivial open problem. The specific case of $\mathrm{SO}(3)$ was treated in \cite{TUCQuad} from which we obtained quadrature rules of degree 1 through 11.

\begin{figure}
    \centering
    \begin{subfigure}{0.35\textwidth}
        \centering
        \includegraphics[height=3.8cm]{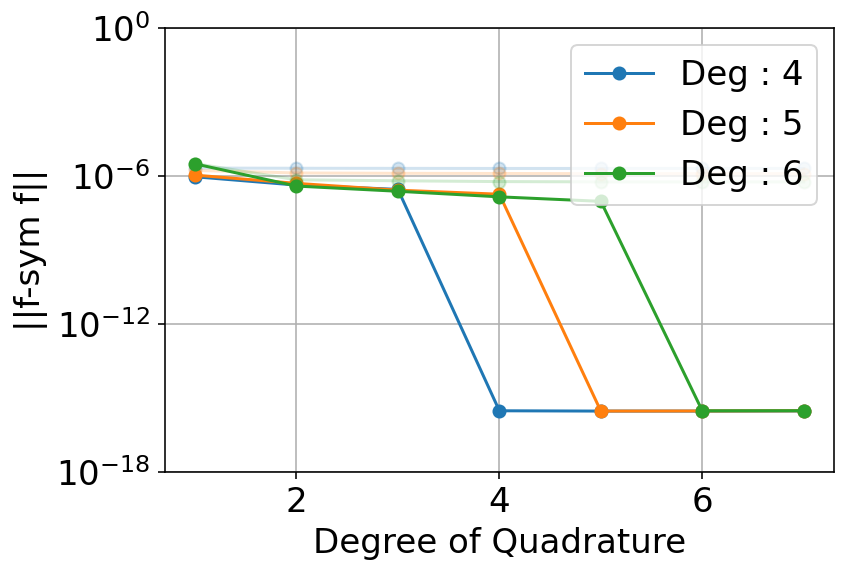} 
        \caption{$UUU$}
        \label{fig:8a}
    \end{subfigure}
    \begin{subfigure}{0.3\textwidth}
        \centering
        \includegraphics[height=3.7cm]{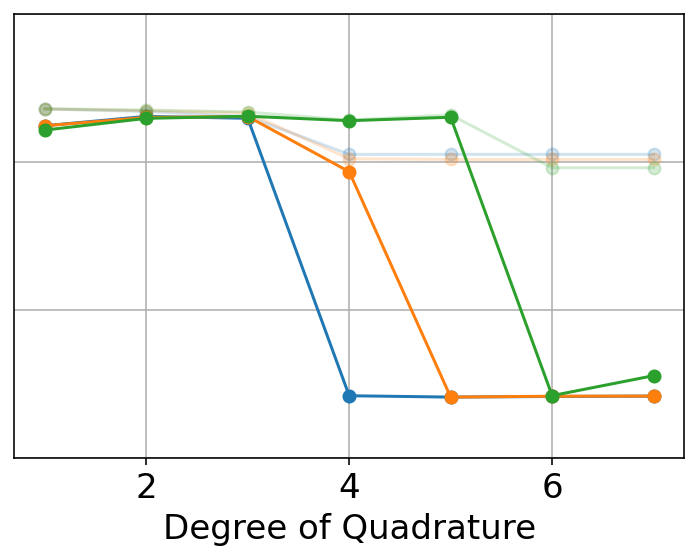} 
        \caption{$\delta UU$}
        \label{fig:8b}
    \end{subfigure}
    \begin{subfigure}{0.3\textwidth}
        \centering
        \includegraphics[height=3.7cm]{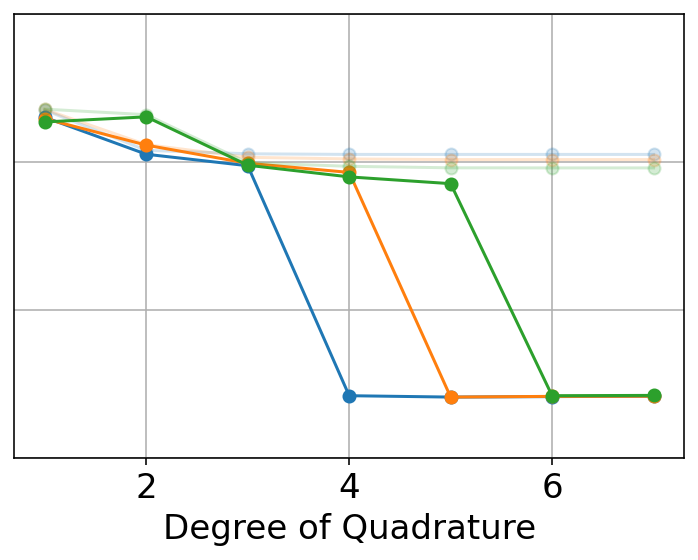} 
        \caption{$\delta^* UU$}
        \label{fig:8c}
    \end{subfigure}
    \vspace{1em}
        \begin{subfigure}{0.35\textwidth}
        \centering
        \includegraphics[height=3.8cm]{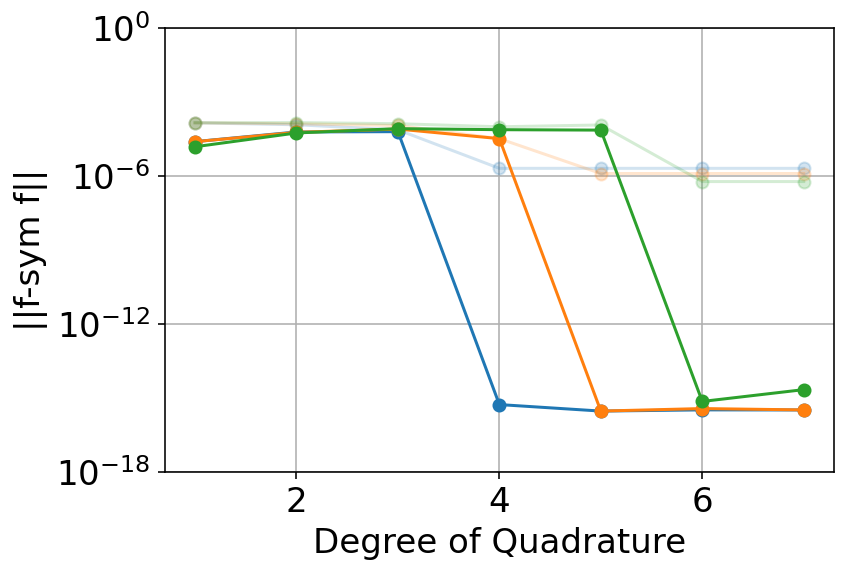} 
        \caption{$\delta H_1U$}
        \label{fig:8d}
    \end{subfigure}
        \begin{subfigure}{0.3\textwidth}
        \centering
        \includegraphics[height=3.7cm]{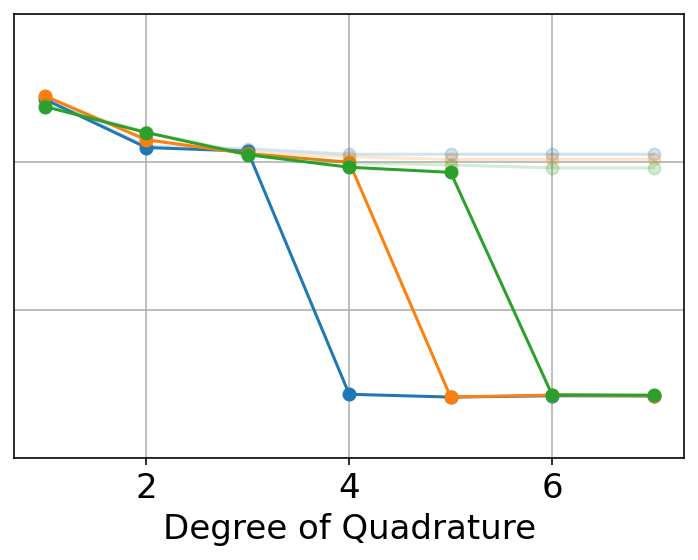} 
        \caption{$\delta^*H_1^*U$}
        \label{fig:8e}
    \end{subfigure}

    \caption{The symmetrisation error evolution for a quadrature augmentation. We approximate a target of degree 11 and the (unaugmented) training set size is 800 and the validation set size is 200. All figures share the same y-axis. The shaded lines correspond to the approximation error.}
    \label{fig:8}
\end{figure}

In Figure \ref{fig:8} we observe (consistent with the 2D tests) that a quadrature augmentation \eqref{eq:lsqaug} leads to perfect symmetry learning once the order of the quadrature exceeds the order of the model. Particularly noteworthy is that while $\delta H_1U$ did not require prior regularisation for an unaugmented LSQ, when the data is not sufficiently augmented the LSQ becomes highly unstable similarly to 2D. We conjecture that the constrained data leaves some basis elements undetermined until the rotations bring enough variety into the data set.

\begin{figure}[!htb]
    \centering
    \begin{subfigure}{0.35\textwidth}
        \centering
        \includegraphics[height=3.8cm]{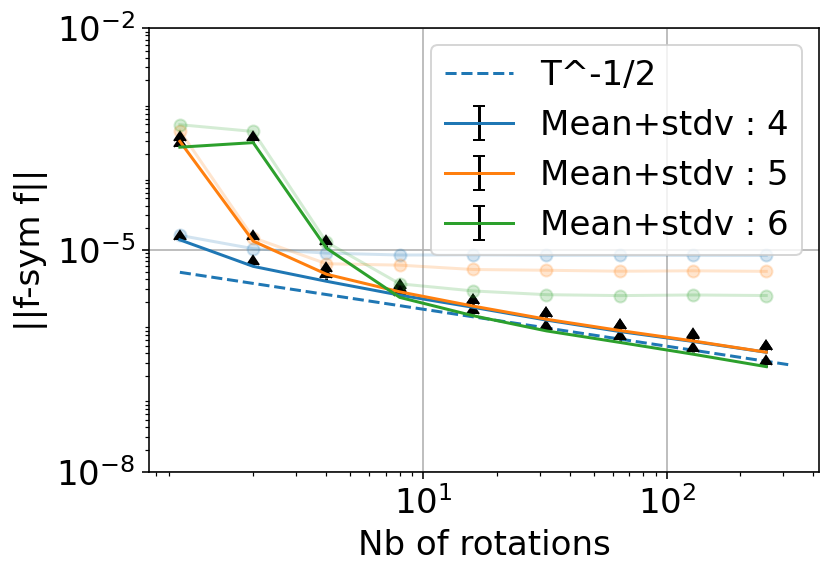} 
        \caption{$UUU$}
        \label{fig:MCpsym3dD1}
    \end{subfigure}
    \begin{subfigure}{0.3\textwidth}
        \centering
        \includegraphics[height=3.7cm]{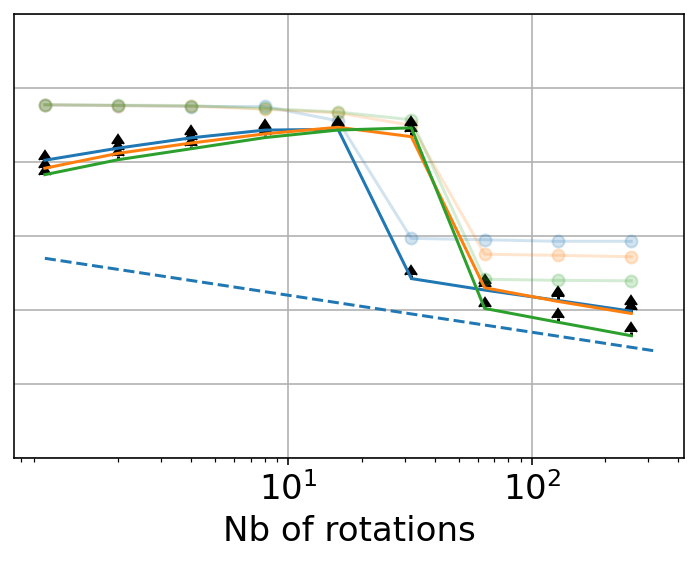} 
        \caption{$\delta UU$}
        \label{fig:MCpsym3dD2}
    \end{subfigure}
    \begin{subfigure}{0.3\textwidth}
        \centering
        \includegraphics[height=3.7cm]{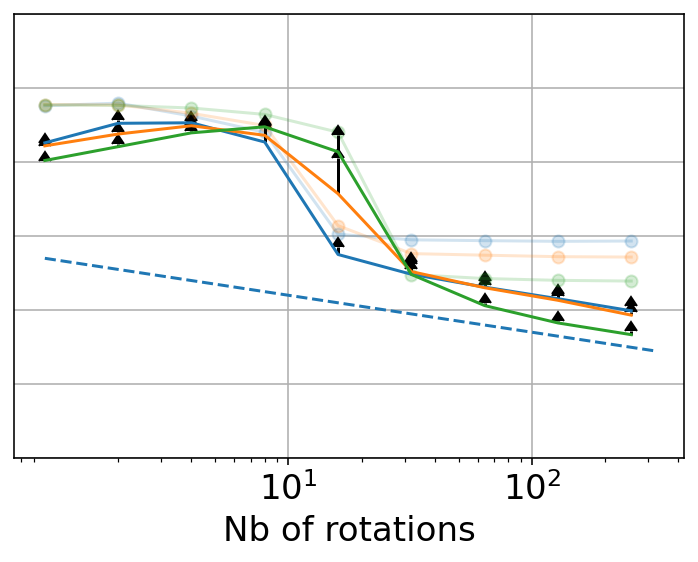} 
        \caption{$\delta^* UU$}
        \label{fig:MCpsym3dD3}
    \end{subfigure}
    \vspace{1em}
        \begin{subfigure}{0.35\textwidth}
        \centering
        \includegraphics[height=3.8cm]{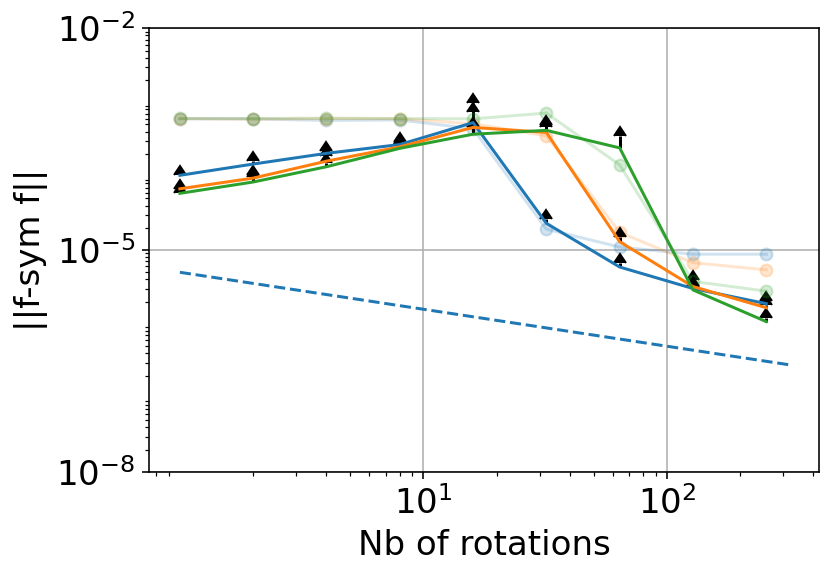} 
        \caption{$\delta H_1U$}
        \label{fig:MCpsym3dD4}
    \end{subfigure}
        \begin{subfigure}{0.3\textwidth}
        \centering
        \includegraphics[height=3.7cm]{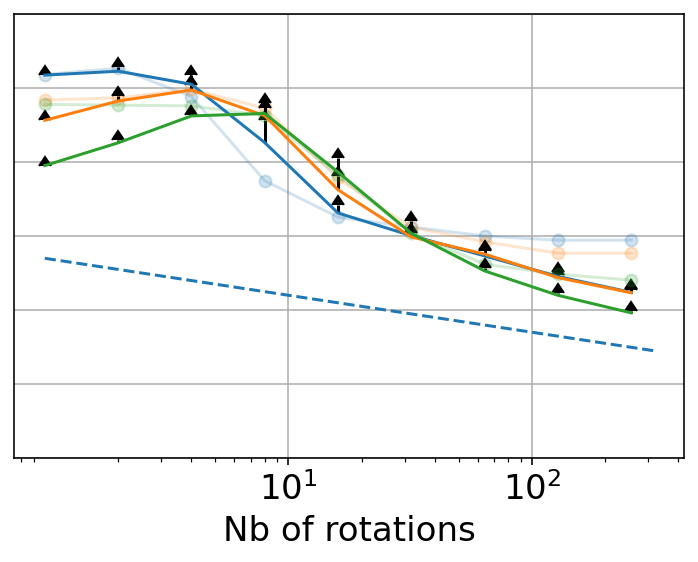} 
        \caption{$\delta^* H_1^*U$}
        \label{fig:MCpsym3dD5}
    \end{subfigure}

    \caption{Symmetrisation error evolution for random augmentations. We approximate a target of degree 11 and the training set size is 800 and the validation set size is 200. All schemes except for $UUU$ are regularised. All figures share the same y-axis. We run the 10 independent simulations before plotting their mean and standard deviation symmetrisation errors. The shaded lines correspond the approximation error of the first trial.}
    \label{fig:MCpsym3d}
\end{figure}

For comparison, Figure \ref{fig:MCpsym3d} shows the evolution of the symmetrisation error under a random augmentation \eqref{eq:lsqaug}. We observe the same behaviour as for $d=1$. The expected square root decay is observed as well as the drops in $\esym$ once enough rotations have been applied.

\begin{figure}[!htb]
    \centering
    \begin{subfigure}{0.35\textwidth}
        \centering
        \includegraphics[height=3.8cm]{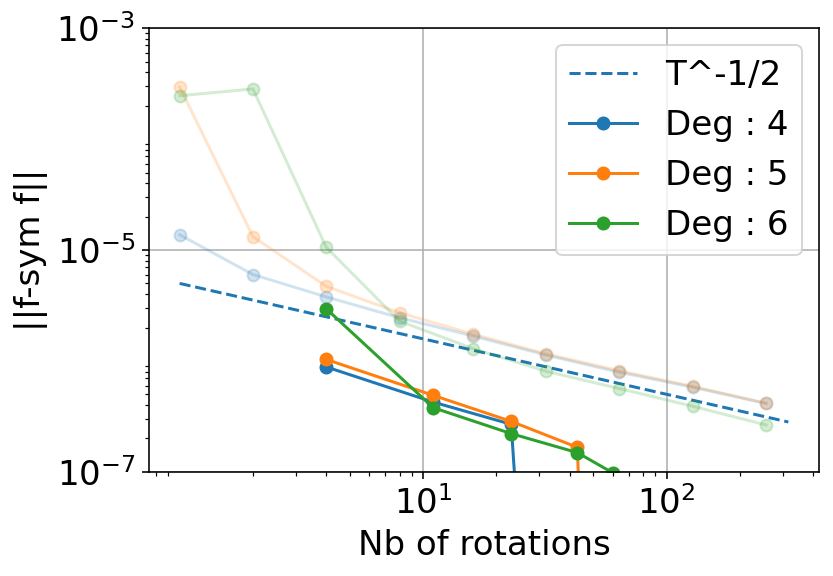} 
        \caption{$UUU$}
        \label{fig:compare3dD1}
    \end{subfigure}
    \begin{subfigure}{0.3\textwidth}
        \centering
        \includegraphics[height=3.7cm]{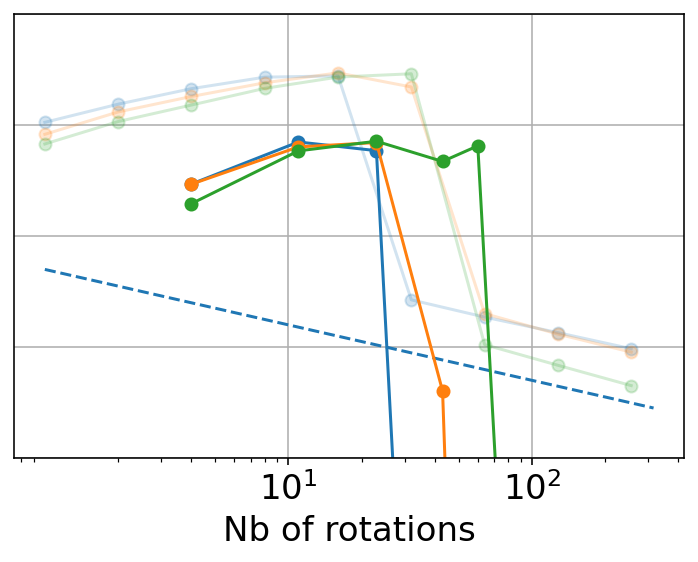} 
        \caption{$\delta UU$}
        \label{fig:compare3dD2}
    \end{subfigure}
    \begin{subfigure}{0.3\textwidth}
        \centering
        \includegraphics[height=3.7cm]{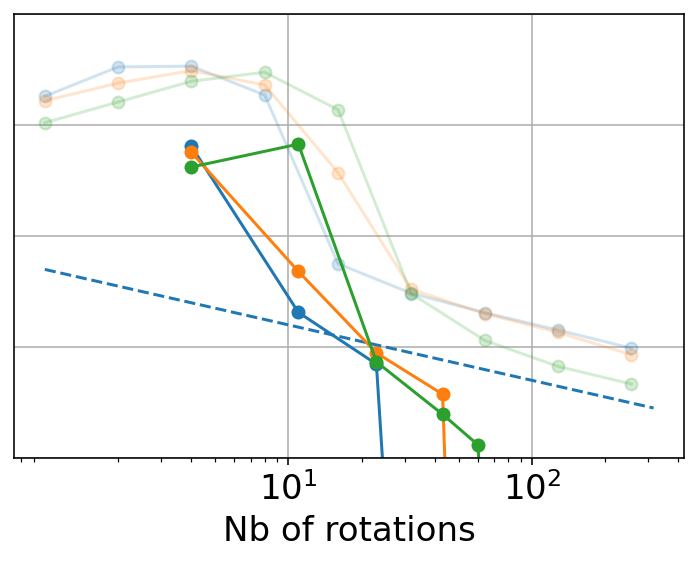} 
        \caption{$\delta^* UU$}
        \label{fig:compare3dD3}
    \end{subfigure}
    \vspace{1em}
        \begin{subfigure}{0.35\textwidth}
        \centering
        \includegraphics[height=3.8cm]{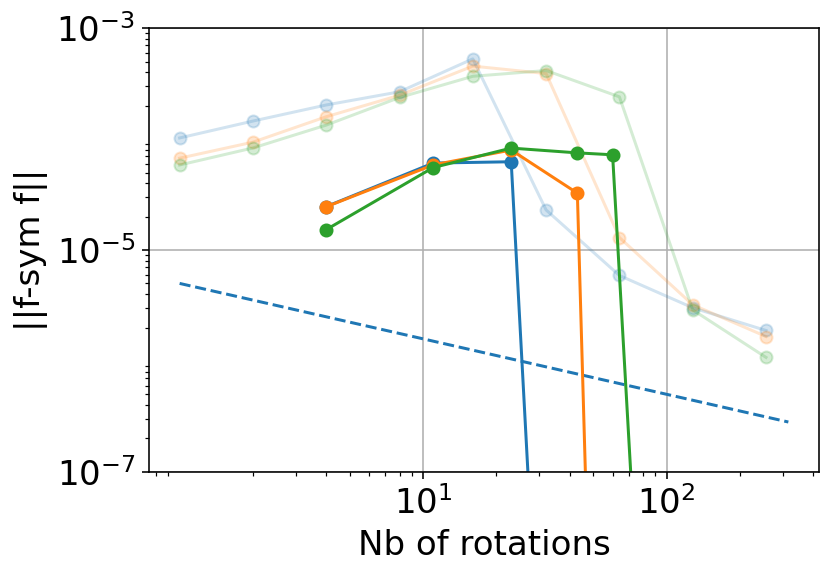} 
        \caption{$\delta H_1U$}
        \label{fig:compare3dD4}
    \end{subfigure}
        \begin{subfigure}{0.3\textwidth}
        \centering
        \includegraphics[height=3.7cm]{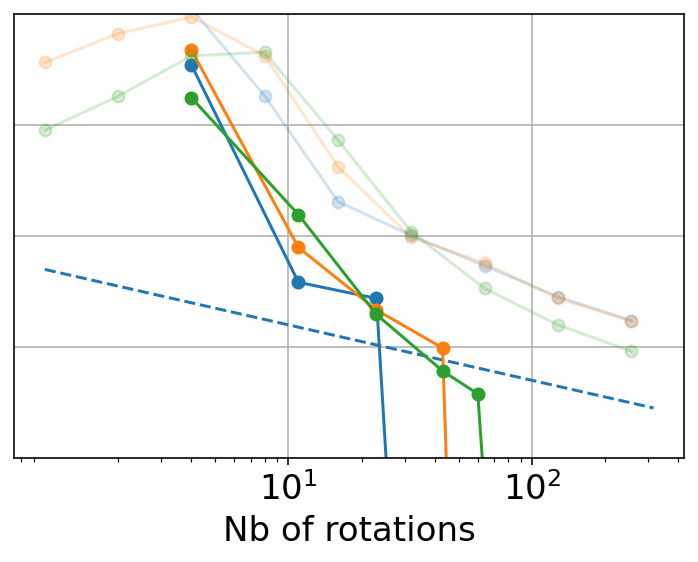} 
        \caption{$\delta^* H_1^*U$}
        \label{fig:compare3dD5}
    \end{subfigure}

    \caption{The superposition of Figures \ref{fig:MCpsym3d} and \ref{fig:8}. The shaded plots correspond to random augmentation while the solid ones correspond to quadrature augmentation. All plots share the same y-axis.}
    \label{fig:compare3d}
\end{figure}

Figure \ref{fig:compare3d} compares the symmetrisation error $\esym$ of both random and quadrature augmentation \eqref{eq:lsqaug}. Just as in Figure \ref{fig:compare2d} we observe that empirically quadrature augmentation is an order of magnitude better than random augmentations. Additionally the square root behaviour can be seen in Figure \ref{fig:compare3dD1}.

\newpage

\section{Conclusion}
We considered the approximation on functions invariant under a group action, in a data-driven setting. We analyzed two methods for augmenting a dataset and least squares system in order to improve the symmetry properties of the resulting approximant: a deterministic quadrature based approach and a probabilistic approach. For both cases we provided rigorous convergence rates. The key take-away is that, when the original data distribution is highly concentrated relative to the symmetry group, then the quadrature approach can provide exact symmetrization but only at the cost of a significant enlargement of the least squares system. In the pre-asymptotic regime it provides no discernable improvement over the probabilistic approach, which unsurprisingly converges at the usual $T^{-1/2}$ rate with $T$ begin the number of augmentations (though the explanation for that rate was not immediately obvious). Neither of our approaches make data augmentation both accurate and computationally cheap at the same time.

Following our arguments presented in Section \ref{section:AngMom}, when conservation properties are important in a downstream simulation task, then preserving symmetry as tightly as possible can be crucial and in those situations exact symmetrization, whether by enforcing it in the architecture, or through the proposed quadrature approach, is advisable.

We conclude with the caveat, that the present work is restricted to polynomial models and that the extension of our observations to nonlinear and deep learning methods are not obvious: further theory and empirical tests are required but empirical evidence suggests that quadrature augmentation has potential in that setting as well as it performs an order of magnitude better than random augmentations in our numerical tests.

\section*{Acknowledgements}
We are grateful to stimulating discussions with Timon Gutleb and Isaac Holt at early stages of this research project. The authors were supported by the Natural Sciences and Engineering Research Council of Canada through NSERC Discovery grants and an NSERC-NSF Collaboration on quantum science and artificial intelligence grant. The research was also supported through computational resources and services provided by the Digital Research Alliance of Canada (alliancecan.ca) and by Advanced Research Computing at the University of British Columbia. Christoph Ortner is a partner in Symmetric Group LLP which licenses force fields commercially.

\section{Appendix}
\label{Appendix}
In this appendix we discuss the observed oddity in Figure \ref{fig:Augpsym2d}. We ran simulations, sampling form the uniform distribution, for target functions of different regularity. We simulate targets with coefficients that asymptotically behave like $ K^{-1/2}, K^{-1}, K^{-2}, K^{-3}$, which corresponds to $\frac{1}{2}$-Hölder, $C^1$ and $C^2$ and $C^3$ functions.

\begin{figure}
    \centering
    \begin{subfigure}{0.4\textwidth}
    \centering
    \includegraphics[width=\textwidth]{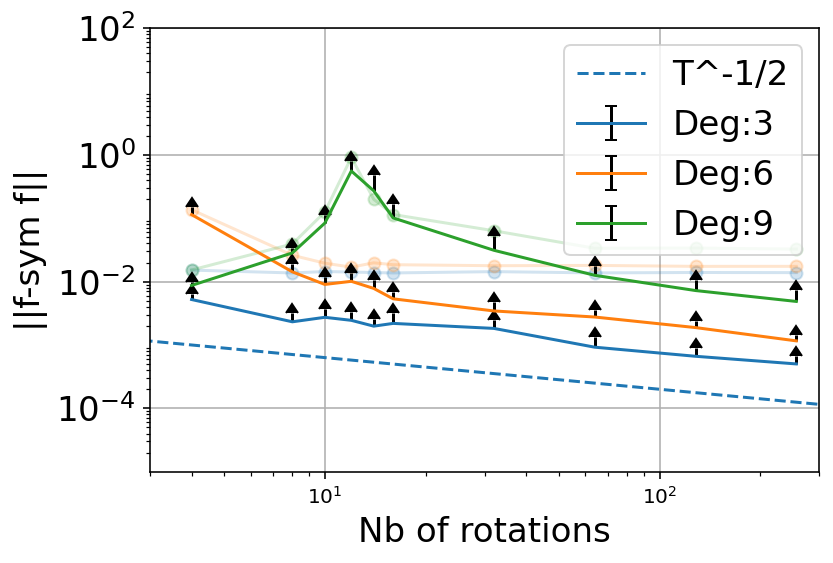}
    \caption{$1/2$-Hölder}
    \label{fig:A10.5H}
    \end{subfigure}
    \begin{subfigure}{0.4\textwidth}
    \centering
    \includegraphics[width=\textwidth]{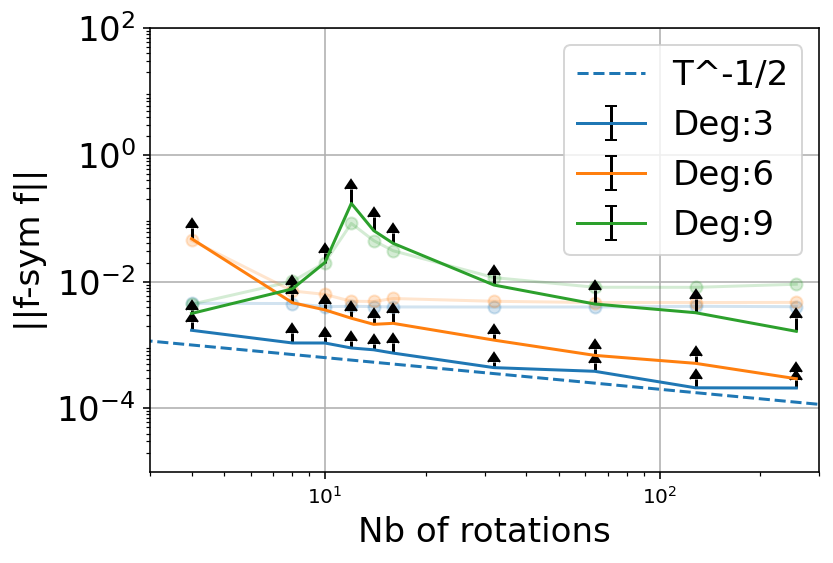}
    \caption{$C^1$}
    \label{fig:A1C1}
    \end{subfigure}
    \vspace{1em}
    \begin{subfigure}{0.4\textwidth}
    \centering
    \includegraphics[width=\textwidth]{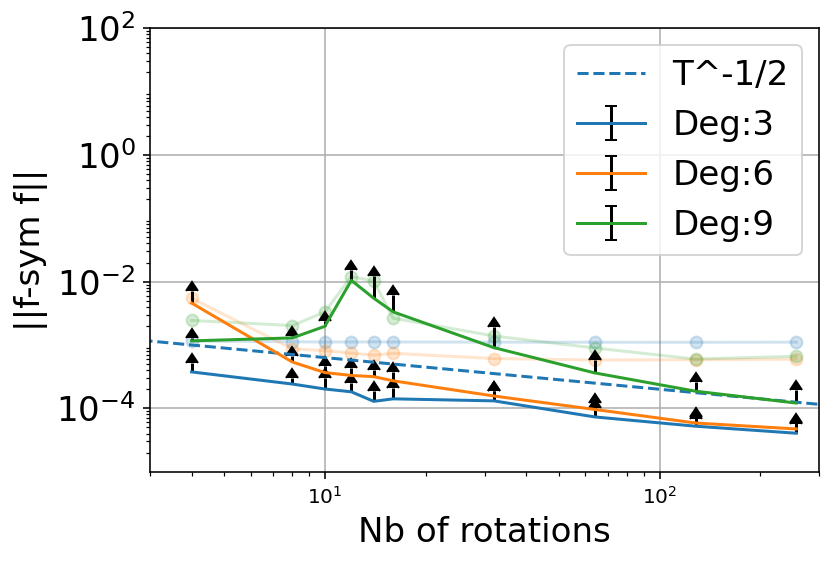}
    \caption{$C^2$}
    \label{fig:A1C2}
    \end{subfigure}
    \begin{subfigure}{0.4\textwidth}
    \centering
    \includegraphics[width=\textwidth]{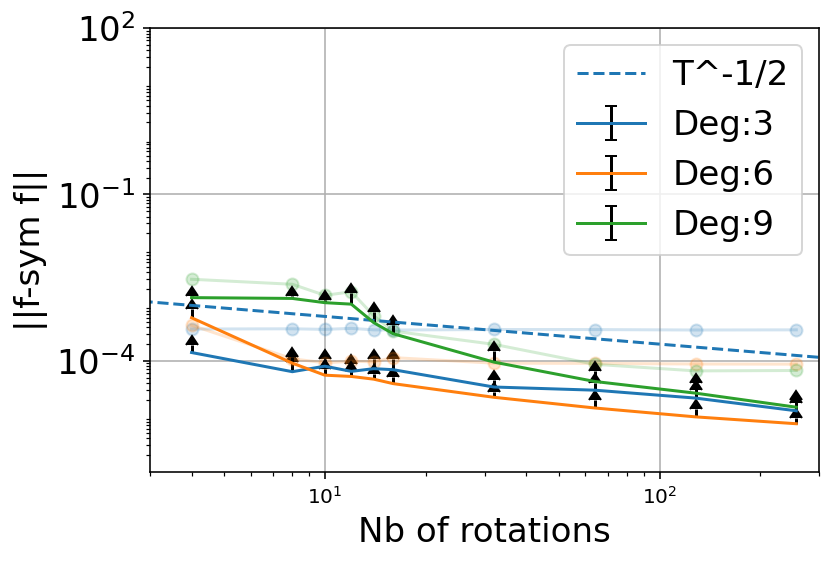}
    \caption{$C^3$}
    \label{fig:A1C3}
    \end{subfigure}
    \caption{$\varepsilon_{\mathrm{psym}}$ for varying regularity in the target function. The unaugmented training set contained 100 points, the validation set size is 100 and the target polynomial is of constant degree 30.}
    \label{fig:A1}
\end{figure}
From Figure \ref{fig:A1} we observe the effects of the upper bound \eqref{eq:SchurUB}. Indeed, we can see that for less regular target functions, our aforementioned intuition was correct, it is easier to symmetrise a low degree approximant. But then, as the regularity increases, the higher degree approximants catch up and eventually overtake the lower degree ones. Showing that regularity makes symmetrisation easier.

\newpage

\bibliographystyle{unsrtnat}
\bibliography{ref.bib}   
\end{document}